\documentclass[11pt,reqno]{amsart}
\usepackage[utf8]{inputenc}

\usepackage{amsfonts, amssymb, amsmath, amsthm, mathtools}
\usepackage{mathrsfs} 
\usepackage{latexsym}
\usepackage{enumerate}
\usepackage{multicol}
\usepackage{times}
\usepackage{verbatim}
\usepackage{tikz-cd}
\usepackage{fullpage}
\usepackage{here}
\usepackage{url}
\usepackage{csquotes}
\usepackage{placeins}
\usepackage{epic}
\usepackage{graphicx}
\usepackage{epstopdf}
\usepackage{pstricks}
\usepackage{pst-plot}
\usepackage{tikz}
\usepackage{xcolor}
\usepackage{graphicx}
\usepackage{subcaption}

\usetikzlibrary{shapes.geometric, positioning, calc}

\input xypic
\xyoption{all}
\xyoption{poly}

\usepackage[colorlinks=true]{hyperref}
\hypersetup{citecolor=blue, linkcolor=blue}

\usepackage[capitalise]{cleveref}

\crefname{equation}{}{}
\crefname{figure}{{\sc Figure}}{{\sc Figure}}
\crefname{subsection}{Subsection}{Subsections}

\usetikzlibrary{matrix,arrows,decorations.pathmorphing}

\usepackage[euler]{textgreek}
\usepackage{tikz,tikz-cd}
\usepackage[colorinlistoftodos, textwidth = 2.3cm]{todonotes}
\newcommand{\nc}{\newcommand}

\newtheorem{theorem}{Theorem}[section]
\newtheorem{proposition}[theorem]{Proposition}
\newtheorem{lemma}[theorem]{Lemma}
\newtheorem{corollary}[theorem]{Corollary}
\newtheorem{conjecture}[theorem]{Conjecture}

\newtheorem*{claim*}{Claim}

\theoremstyle{definition}
\newtheorem{algorithm}[theorem]{Algorithm}
\newtheorem{example}[theorem]{Example}
\newtheorem{definition}[theorem]{Definition}
\newtheorem{remark}[theorem]{Remark}

\newcommand{\C}{{\mathbb C}}
\newcommand{\F}{{\mathbb F}}
\newcommand{\p}{{\mathfrak p}}
\newcommand{\rad}{\operatorname{rad}}
\newcommand{\Z}{{\mathbb Z}}
\newcommand{\Tr}{\operatorname{Tr}}

\usepackage{tabstackengine}
\stackMath

\numberwithin{equation}{section} 
\numberwithin{figure}{section}
\numberwithin{table}{section}

\nc{\Beaver}[1]{\todo[size=\tiny,color=cyan!10]{#1 \\ \hfill --- Beaver}}
\nc{\BEAVER}[1]{\todo[size=\tiny,inline,color=cyan!10]{#1
		\\ \hfill --- Beaver}}
  
\nc{\Semin}[1]{\todo[size=\tiny,color=magenta!10]{#1 \\ \hfill --- Semin}}
\nc{\SEMIN}[1]{\todo[size=\tiny,inline,color=magenta!10]{#1
		\\ \hfill --- Semin}}

\nc{\nt}[1]{\todo[size=\tiny,color=exgreen!10]{#1 \\ \hfill --- Note}}
\nc{\NT}[1]{\todo[size=\tiny,inline,color=exgreen!10]{#1
		\\ \hfill --- Note}}

\title{Multiplicative structure of shifted multiplicative subgroups \\ and its applications to Diophantine tuples}
\author{Seoyoung Kim}
\address{Departement
Mathematik und Informatik, Universit\"at Basel, Spiegelgasse 1, 4051 Basel, Switzerland}
\email{seoyoung.kim@unibas.ch}
\author{Chi Hoi Yip}
\address{School of Mathematics\\ Georgia Institute of Technology\\ GA 30332\\ United States}
\email{cyip30@gatech.edu}
\author{Semin Yoo}
\address{Discrete Mathematics Group \\ Institute for Basic Science \\ 55 Expo-ro Yuseong-gu, Daejeon 34126 \\ South Korea}
\email{syoo19@ibs.re.kr}

\begin{document}

\begin{abstract}
In this paper, we investigate the multiplicative structure of a shifted multiplicative subgroup and its connections with additive combinatorics and the theory of Diophantine equations. Among many new results, we highlight our main contributions as follows. First, we show that if a nontrivial shift of a multiplicative subgroup $G$ contains a product set $AB$, then $|A||B|$ is essentially bounded by $|G|$, refining a well-known consequence of a classical result by Vinogradov. Second, we provide a sharper upper bound of $M_k(n)$, the largest size of a set such that each pairwise product of its elements is $n$ less than a $k$-th power, refining the recent result of Dixit, Kim, and Murty.  One main ingredient in our proof is the first non-trivial upper bound on the maximum size of a generalized Diophantine tuple over a finite field. In addition, we determine the maximum size of an infinite family of generalized Diophantine tuples over finite fields with square order, which is of independent interest.  
We also make significant progress towards a conjecture of S\'{a}rk\"{o}zy on the multiplicative decompositions of shifted multiplicative subgroups. In particular, we prove that for almost all primes $p$, the set $\{x^2-1: x \in \F_p^*\} \setminus \{0\}$ cannot be decomposed as the product of two sets in $\F_p$ non-trivially.
\end{abstract}

\maketitle


\section{Introduction}

Let $q$ be a prime power, and let $\F_q$ be the finite field with $q$ elements. Let $G$ be a multiplicative subgroup of $\F_q$. While $G$ itself has a ``perfect" multiplicative structure, it is natural to ask if a (non-trivial additive) shift of $G$ still possesses some multiplicative structure. Indeed, as a fundamental question in additive combinatorics, this question has drawn the attention of many researchers and it is closely related to many questions in number theory. For example, a classical result of Vinogradov \cite{V54} states that for a prime $p$ and an integer $n$ such that $p \nmid n$, if $A, B \subset \{1,2,\ldots, p-1\}$, then
 \begin{equation}\label{Vinogradov}
 \bigg|\sum_{a\in A,\, b\in B} \bigg(\frac{ab+n}{p}\bigg)\bigg|  \leq \sqrt{p|A||B|}.
 \end{equation}
More generally, inequality~\cref{Vinogradov} extends to all nontrivial multiplicative characters over all finite fields; see \cref{charsum:sym}. Inequality~\cref{Vinogradov} leads an estimate on the size of a product set contains in the set of shifted squares: if $A,B \subset \F_p^*$, $\lambda \in \F_p^*$, and $G$ is the subgroup of $\F_p^*$ of index $2$ such that $AB \subset (G+\lambda)$, then 
\begin{equation}\label{V-app}
|A||B|<(1+o(1))p.    
\end{equation}
For more recent works related to this question and its connection with other problems, we refer to \cite{G01, MSS18,S14, VS12} and references therein. An {\it analogue} of this question over integers is closely related to the well-studied Diophantine tuples and their generalizations; see \cref{1.1}.  

In this paper, we study the multiplicative structure of a shifted multiplicative subgroup following the spirit of the aforementioned works and discuss a few new applications in additive combinatorics and Diophantine equations. More precisely, one of our contributions is the following theorem.
 
\begin{theorem}\label{thm: stepanovea}
Let $d \mid (q-1)$ with $d \ge 2$. Let $S_d=\{x^d: x \in \F_q^*\}$. Let $A,B \subset \F_q^*$ and $\lambda \in \F_q^*$ with $|A|, |B| \geq 2$. 
 Assume further that $\binom{|A|-1+\frac{q-1}{d}}{|A|} \not \equiv 0 \pmod p$ if $q \neq p$. If $AB+\lambda \subset S_d \cup \{0\}$, then $$|A||B| \leq |S_d|+|B \cap (-\lambda A^{-1})|+|A|-1.$$ 
Moreover, when $\lambda \in S_d$, we have a stronger upper bound:
$$|A||B| \leq |S_d|+|B \cap (-\lambda A^{-1})|-1.$$
\end{theorem}

Clearly, \cref{thm: stepanovea} improves inequality~\cref{V-app} implied by Vinogradov's estimate \cref{Vinogradov} when $d=2$, and the generalization of inequality~\cref{V-app} by Gyarmati \cite[Theorem 8]{G01} for general $d$, where the upper bound is given by $(\sqrt{p}+2)^2$ when $q=p$ is a prime. We remark that in general the condition on the binomial coefficient in the statement of \cref{thm: stepanovea} cannot be dropped when $q$ is not a prime; see \cref{mainthm3}.

The proof of \cref{thm: stepanovea} is based on Stepanov's method \cite{S69}, and is motivated by a recent breakthrough of Hanson and Petridis \cite{HP}. In fact, 
\cref{thm: stepanovea} can be viewed as a multiplicative analog of their results. 
Going beyond the perspective of these multiplicative analogs, we provide new insights into the application of Stepanov's method. 
For example, our technique applies to all finite fields while their technique only works over prime fields. We also prove a similar result for restricted product sets (see \cref{thm: restricted}), whereas their technique appears to only lead to a weaker bound; see \cref{remark:weakerbound}.

Besides \cref{thm: stepanovea}, we also provide three novel applications of \cref{thm: stepanovea} and its variants. These applications significantly improve on many previous results in the literature. Unsurprisingly, to achieve these applications, we need additional tools and insights from Diophantine approximation, sieve methods, additive combinatorics, and character sums.  From here, we briefly mention what applications are about. In \cref{1.1}, we improve the upper bound of generalized Diophantine tuples over integers.
Interestingly, \cref{thm: stepanovea} is closely related to a bipartite version of Diophantine tuples over finite fields. This new perspective yields a substantial improvement in the result of generalized Diophantine tuples over integers. 
In \cref{1.2}, we obtain the first nontrivial upper bounds on generalized Diophantine tuples and strong Diophantine tuples over finite fields. Moreover, some of our new bounds are sharp. Last but not least, in \cref{1.3}, we make significant progress towards a conjecture of S\'{a}rk\"{o}zy \cite{S14} on multiplicative decompositions of shifted multiplicative subgroups.
We elaborate on the context of these applications in the next subsections.

\subsection{Diophantine tuples over integers}\label{1.1}
A set $\{a_{1}, a_{2},\ldots, a_{m}\}$ of distinct positive integers is a \textit{Diophantine $m$-tuple} if the product of any two distinct elements in the set is one less than a square. 
The first known example of integral Diophantine $4$-tuples is $\{1,3,8,120\}$ which was studied by Fermat. The Diophantine $4$-tuple was extended by Euler to the rational $5$-tuple $\{1, 3, 8, 120, \frac{777480}{8288641}\}$, and it had been conjectured that there is no Diophantine $5$-tuple. The difficulty of extending Diophantine tuples can be explained by its connection to the problem of finding integral points on elliptic curves: if $\{a,b,c\}$ forms a Diophantine $3$-tuple, in order to find a positive integer $d$ such that $\{a,b,c,d\}$ is a Diophantine $4$-tuple, we need to solve the following simultaneous equation for $d$:
\[ad+1=s^2, \quad bd+1=t^2, \quad cd+1=r^2.\]
This is related to the problem of finding an integral point $(d,str)$ on the following elliptic curve
\[y^2=(ax+1)(bx+1)(cx+1).\]
From this, we can deduce that there are no infinite Diophantine $m$-tuples by Siegel's theorem on integral points. On the other hand, Siegel's theorem is not sufficient to give an upper bound on the size of Diophantine tuples due to its ineffectivity. In the same vein, finding a Diophantine tuple of size greater than or equal to $5$ is related to the problem of finding integral points on hyperelliptic curves of genus $g\geq 2$. Despite the aforementioned difficulties, the conjecture on the non-existence of Diophantine $5$-tuples was recently proved to be true in the sequel of important papers by Dujella \cite{duj2004diophantine}, and He, Togb\'e, and Ziegler \cite{HTZ19}.

The definition of Diophantine $m$-tuples has been generalized and studied in various contexts. We refer to the recent book of Dujella \cite{D24} for a thorough list of known results on the topic and their reference. In this paper, we focus on the following generalization of Diophantine tuples: for each $n \ge 1$ and $k \ge 2$, we call a set $\{a_{1},a_{2},\ldots, a_{m}\}$ of distinct positive integers a \textit{Diophantine $m$-tuple with property $D_{k}(n)$} if the product of any two distinct elements is $n$ less than a $k$-th power. We write
\[M_{k}(n)=\sup \{|A| \colon A\subset{\mathbb{N}} \text{ satisfies the property }D_{k}(n)\}.\]
Similar to the classical case, the problem of finding $M_{k}(n)$ for $k\geq 3$ and $n \geq 1$ is related to the problem of counting the number of integral points of the superelliptic curve
\begin{equation*}
    y^k =f(x)=(a_1 x +n)(a_2 x +n) (a_3 x+n)
\end{equation*}
The theorem of Faltings \cite{faltings} guarantees that the above curve has only finitely many integral points, and this, in turn, implies that a set with property $D_k(n)$ must be finite. The known upper bounds for the number of integral points depend on the coefficients of $f(x)$. The Caporaso-Harris-Mazur conjecture \cite{CHM} implies that $M_k(n)$ is uniformly bounded, independent of $n$. For other conditional bounds, we refer the readers to \cref{sec:PaleyABC}. 

Unconditionally, in \cite{BD03}, Bugeaud and Dujella \cite[Corollary 4]{BD03} showed that $M_3(1) \leq 7$, $M_k(1) \leq 5$ for $k \in \{4,5\}$, $M_k(1) \leq 4$ for $6 \leq k \leq 176$, and the uniform bound $M_{k}(1) \leq 3$ for any $k \geq 177$ \footnote{As pointed out by \cite{BDHL11}, there was a minor inaccuracy in the original proof of \cite[Corollary 4]{BD03}, but it only affected the upper bound on $M_5(1)$.}. On the other hand, the best-known upper bound on $M_2(n)$ is $(2+o(1))\log n$, due to the second author~\cite{Y24+}. Very recently, Dixit, Murty, and the first author \cite{DKM22} studied the size of a generalized Diophantine $m$-tuple with property $D_k(n)$, improving the previously best-known upper bound $M_3(n)\leq 2|n|^{17}+6$ and $M_k(n) \leq 2|n|^5 + 3$ for $k\geq 5$ given by B\'{e}rczes, Dujella, Hajdu and Luca \cite{BDHL11} when $n \to \infty$. They showed that if $k$ is fixed and $n \to \infty$, then $M_k(n)\ll_k \log n$. Following their proof in \cite{DKM22}, the bound can be more explicitly expressed as $M_k(n)\leq (3\phi(k)+o(1))\log n$ when $k \geq 3$ is fixed, $n \to \infty$, and $\phi$ is the Euler phi function. Note that their upper bound on $M_k(n)$ is perhaps not desirable. Indeed, it is natural to expect that $M_k(n)$ would decrease if $n$ is fixed, and $k$ increases, since $k$-th powers become sparser. Instead, our new upper bounds on $M_k(n)$ support this heuristic.

In this paper, we provide a significant improvement on the upper bound of $M_{k}(n)$ by using a novel combination of Stepanov's method and Gallagher's larger sieve inequality. 
In order to state our first result, we define the following constant 
\begin{equation}
\eta_k=\min_{\mathcal{I}} \frac{|\mathcal{I}|}{T_\mathcal{I}^2}   
\end{equation}
for each $k \geq 2$, where the minimum is taken over all nonempty subsets $\mathcal{I}$ of the set 
$$\{1 \leq i \leq k: \gcd(i,k)=1, \gcd(i-1,k)>1\},$$
and $T_{\mathcal{I}}=\sum_{i \in \mathcal{I}} \sqrt{\gcd(i-1,k)}$. 

\textcolor{black}{
\begin{theorem}\label{mainthm1}
There is a constant $c'>0$, such that as $n \to \infty$,
\begin{equation}
\label{main-upper-bound}
    M_k(n)\leq \bigg(\frac{2k}{k-2}+o(1)\bigg) \ \eta_k  \phi(k) \log n,
\end{equation}
holds uniformly for positive integers $k,n \geq 3$ such that $\log k \leq c'\sqrt{\log \log n}$. 
\end{theorem}}

The constant $\eta_k$ is essentially computed via the optimal collection of ``admissible" residue classes when applying Gallagher's larger sieve (see \cref{sec: GDM}). Note that when $\mathcal{I}=\{1\}$, we have $T_\mathcal{I}=\sqrt{k}$, and hence we have $\eta_k \leq \frac{1}{k}$. \textcolor{black}{In particular, if $k \geq 3$ is fixed and $n \to \infty$, inequality \cref{main-upper-bound} implies the upper bound
\begin{equation}
\label{eq1.5}
M_k(n) \leq \frac{(2+o(1))\phi(k)}{k-2} \log n, 
\end{equation}
which already improves the best-known upper bound $M_k(n)\leq (3\phi(k)+o(1))\log n$ of \cite{DKM22} that we mentioned earlier substantially}. \textcolor{black}{In \cref{appendix}, we illustrate the bound in inequality \cref{main-upper-bound}: for $2 \leq k \leq 1001$, we compute the suggested upper bound
\[\nu_{k} = \frac{2k}{k-2} \eta_k \phi(k)
\]
of $\gamma_{k}=\limsup_{n \to \infty} \frac{M_k(n)}{\log n}$.} From \cref{fig:graph1}, one can compare the bound of $M_{k}(n)$ in  \cref{mainthm1} with the bound in \cite{DKM22}. From inequality \cref{eq1.5}, we see $\gamma_k$ is uniformly bounded by $6$. 
\cref{table: the upper bound1} illustrates better upper bounds on $\gamma_k$ for $2 \le k \leq 201$. In particular, we use a simple greedy algorithm to determine $\eta_k$ for a fixed $k$. We also refer to \cref{sec: approximation} for a simple upper bound on $\eta_k$, which well approximates $\eta_k$ empirically.

At first glance, \cref{mainthm1} improves the bound in \cite{DKM22} of $M_k(n)$ by only a constant multiplicative factor when $k$ is fixed.  Nevertheless, note that \cref{mainthm1} holds uniformly for $k$ and $n$ as long as $\log k \leq c'\sqrt{\log \log n}$. Thus, when $k$ is assumed to be a function of $n$ which increases as $n$ increases, we can break the ``$\log n$ barrier" in \cite{DKM22}, that is, $M_{k}(n)=O_{k}(\log n)$, and provide a \emph{dramatic} improvement. 
\begin{theorem}
\label{mainthm1.5}
There is $k=k(n)$ such that $\log k \asymp \sqrt{\log \log n}$, and 
$$M_k(n)\ll \exp\bigg({-}\frac{c''(\log \log n)^{1/4}}{\log \log \log n}\bigg) \log n,
$$
where $c''>0$ is an absolute constant.
\end{theorem}
The proofs of \cref{mainthm1} and \cref{mainthm1.5} require the study of (generalized) Diophantine tuples over finite fields, which we discuss below.

\subsection{Diophantine tuples over finite fields}\label{1.2}
A \emph{Diophantine $m$-tuple with property $D_{d}(\lambda,\mathbb{F}_{q})$} is a set $\{a_1,\ldots, a_m\}$ of $m$ distinct elements in $\mathbb{F}_{q}^{*}$ such that $a_{i}a_{j}+\lambda$ is a $d$-th power in $\mathbb{F}_q^*$ or $0$ whenever $i\neq j$. Moreover, we also define the strong Diophantine tuples in finite fields motivated by Dujella and Petri\v{c}evi\'{c} \cite{DP08}: a \emph{strong Diophantine $m$-tuple with property $SD_{d}(\lambda,\mathbb{F}_q)$} 
 is a set $\{a_1,\ldots, a_m\}$ of $m$ distinct elements in $\mathbb{F}_{q}^{*}$ such that $a_{i}a_{j}+\lambda$ is a $d$-th power in $\mathbb{F}_q^*$ or $0$ for any choice of $i$ and $j$. Unlike the natural analog for the classical Diophantine tuples (of property $D_{2}(1)$), it makes sense to talk about the strong Diophantine tuples in $\mathbb{F}_q$. The strong generalized Diophantine tuples with property $D_{k}(n)$ in for general $k$ and $n$ are also meaningful to study: the problem of counting the explicit size of the strong generalized Diophantine tuples with property $D_k(n)$ involves the problem of counting solutions of the equations appearing in the statement of Pillai's conjecture. \cref{mainthm1} can be improved for strong generalized Diophantine tuples with property $D_k(n)$; see \cref{mainthm1_strong}. 
 
The generalizations of Diophantine tuples over finite fields are of independent interest. Perhaps the most interesting question to explore is the exact analog of estimating $M_k(n)$ as discussed in \cref{1.1}. Indeed, estimating the size of the largest Diophantine tuple with property $SD_{d}(\lambda,\mathbb{F}_{q})$ or with property $D_{d}(\lambda,\mathbb{F}_{q})$ is of particular interest for the application of Diophantine tuples (over integers) as discussed in \cite{BM19, DKM22, D02, GM20, KYY24a}. Similarly, we denote 
\begin{align*}
MSD_{d}(\lambda,\mathbb{F}_q)&=\sup \{|A| \colon A \subset \F_q^* \text{ satisfies property }SD_{d}(\lambda,\mathbb{F}_q)\}, \quad \text{and}\\
MD_{d}(\lambda,\mathbb{F}_q)&=\sup \{|A| \colon A \subset \F_q^* \text{ satisfies property }D_{d}(\lambda,\mathbb{F}_q)\}.
\end{align*}
Note that when $\lambda=0$, it is trivial that $MSD_{d}(\lambda,\mathbb{F}_q)=MD_{d}(\lambda,\mathbb{F}_q)=\frac{q-1}{d}$. Thus, we always assume $\lambda \neq 0$ throughout the paper. In \cref{sec: Preliminary estimations}, we give an upper bound $\sqrt{q}+O(1)$ of $MSD_{d}(\lambda,\mathbb{F}_q)$ and $MD_{d}(\lambda,\mathbb{F}_q)$. More precisely, we prove the following proposition using a double character sum estimate. We refer to the bounds in the following proposition as the ``trivial" upper bound.

\begin{proposition}[Trivial upper bound]\label{mainprop1}
Let $d \geq 2$ and let $q \equiv 1 \pmod d$ be a prime power. Let $A \subset \F_q^*$ and $\lambda \in \F_q^*$. Then  $MSD_{d}(\lambda,\mathbb{F}_q) \leq \frac{\sqrt{4q-3}+1}{2}$ and $MD_{d}(\lambda,\mathbb{F}_q) \leq \sqrt{q-\frac{11}{4}}+\frac{5}{2}$.
\end{proposition}

For the case $q=p$, similar bounds of \cref{mainprop1} are known previously in \cite{BM19, DKM22, G01}. On the other hand, \cref{mainprop1} gives an almost optimal bound of  $MSD_{d}(\lambda,\mathbb{F}_q)$ and $MD_{d}(\lambda,\mathbb{F}_q)$ when $q$ is a square (\cref{mainthm3}). Our next theorem improves the trivial upper bounds in \cref{mainprop1} by a multiplicative constant factor $\sqrt{1/d}$ or $\sqrt{2/d}$ when $q=p$ is a prime. 

\begin{theorem}\label{mainthm2}
Let $d \geq 2$. Let $p \equiv 1 \pmod d$ be a prime and let $\lambda \in \F_p^*$. Then 
\begin{enumerate}
\item[\textup{(1)}] $MSD_{d}(\lambda,\mathbb{F}_p) \leq \sqrt{(p-1)/d}+1$. Moreover, if $\lambda$ is a $d$-th power in $\F_p^*$, then we have a stronger upper bound:
$$MSD_{d}(\lambda,\mathbb{F}_p)\leq \sqrt{\frac{p-1}{d}-\frac{3}{4}}+\frac{1}{2}.$$  
\item[\textup{(2)}] $MD_{d}(\lambda,\mathbb{F}_p) \leq \sqrt{2(p-1)/d}+4$.
\end{enumerate}
\end{theorem}

We remark that our new bound on $MSD_{d}(\lambda,\mathbb{F}_p)$ is sometimes sharp. For example, we get a tight bound for a prime $p \in \{5,7,11,13,17,23,31,37,41,53,59,61,113\}$ when $d=2$ and $\lambda=1$. See also \cref{mainthm2.5} and \cref{mainthm2.5_remark} for a generalization of \cref{mainthm2} over general finite fields with non-square order under some extra assumptions. 

Nevertheless, in the case of finite fields of square order, we improve \cref{mainprop1} by a little bit under some minor assumptions; see \cref{thm: square}. Surprisingly, this tiny improvement turns out to be sharp for many \emph{infinite families} of $(q,d,\lambda)$. Equivalently, we determine $MD_d(\lambda, \F_q)$ and $MSD_d(\lambda, \F_q)$ \emph{exactly} in those families. In the following theorem, we give a sufficient condition so that $MD_d(\lambda, \F_q)$ and $MSD_d(\lambda, \F_q)$ can be determined explicitly.

\begin{theorem}\label{mainthm3}
Let $q$ be a prime power and a square, $d \geq 2$, and $d \mid (\sqrt{q}+1)$. 
 Let $S_d=\{x^d: x \in \F_q^*\}$. Suppose that there is $\alpha \in \F_q$ such that $\alpha^2 \in S_d$ and $\lambda \in \alpha^2 \F_{\sqrt{q}}^*$ (for example, if $\alpha=1$ and $\lambda \in \F_{\sqrt{q}}^*)$. Suppose further that $r\leq (p-1)\sqrt{q}$, where $r$ is the remainder of $\frac{q-1}{d}$ divided by $p\sqrt{q}$. Then $MSD_d(\lambda, \F_q)=\sqrt{q}-1$. If $q \geq 25$ and $d \geq 3$, then we have the stronger conclusion that $MD_d(\lambda, \F_q)=\sqrt{q}-1$.
\end{theorem}

Under the assumptions on \cref{mainthm3}, $\alpha \F_{\sqrt{q}}^*$ satisfies the required property $SD_{d}(\lambda,\mathbb{F}_q)$  and $D_{d}(\lambda,\mathbb{F}_q)$. Compared to \cref{mainthm2}, it is tempting to conjecture that such an algebraic construction (which is unique to finite fields with proper prime power order) is the optimal one with the required property. Given \cref{mainprop1}, to show such construction is optimal, it suffices to rule out the possibility of a Diophantine tuple with property $SD_{d}(\lambda,\mathbb{F}_q)$ and $D_{d}(\lambda,\mathbb{F}_q)$ of size $\sqrt{q}$. While this seems easy, it turned out that this requires non-trivial efforts. 

Next, we give concrete examples where \cref{mainthm3} applies.
\begin{example}
Note that a Diophantine tuple with property $SD_2(1, \F_q)$ corresponds to a strong Diophantine tuple over $\F_q$. If $q$ is an odd square, \cref{mainthm3} implies that the largest size of a strong Diophantine tuple over $\F_q$ is given by $\sqrt{q}-1$, which is achieved by $\F_{\sqrt{q}}^*$. Note that in this case we have $r=\frac{p\sqrt{q}-1}{2}<(p-1)\sqrt{q}$. 

We also consider the case that $d=3$, $d \mid (\sqrt{q}+1)$, and $\lambda=1$. In this case, \cref{mainthm3} also applies. Note that $3 \mid (\sqrt{q}+1)$ implies that $p \equiv 2 \pmod 3$, in which case the base-$p$ representation of $\frac{q-1}{3}$ only contains the digit $\frac{p-2}{3}$ and $\frac{2p-1}{3}$, so that the condition $r \leq (p-1)\sqrt{q}$ holds.
\end{example}

One key ingredient of the proof of \cref{mainthm2} and \cref{mainthm3} is \cref{thm: stepanovea}. Indeed, \cref{thm: stepanovea} can also be viewed as on an upper bound of a bipartite version of Diophantine tuples over finite fields. For the applications to strong Diophantine tuples, \cref{thm: stepanovea} is sufficient. On the other hand, to obtain upper bounds on Diophantine tuples (which are not necessarily strong Diophantine tuples), we also need a version of \cref{thm: stepanovea} for restricted product sets, which can be found as \cref{thm: restricted}. Indeed, \cref{thm: stepanovea} alone only implies a weaker bound of the form $2\sqrt{p/d}$ for $MSD_d(\lambda, \F_p)$; see \cref{remark:weakerbound}.

\subsection{Multiplicative decompositions of shifted multiplicative subgroups} \label{1.3}
A well-known conjecture of S\'{a}rk\"{o}zy \cite{S12} asserts that the set of nonzero squares $S_2=\{x^2: x \in \F_p^*\} \subset \F_p$ cannot be written as $S_2=A+B$, where $A,B \subset \F_p$ and $|A|, |B| \geq 2$, provided that $p$ is a sufficiently large prime. This conjecture essentially predicts that the set of quadratic residues in a prime field cannot have a rich additive structure. Similarly, one expects that any non-trivial shift of $S_2$ cannot have a rich multiplicative structure. Indeed, this can be made precise via another interesting conjecture of S\'{a}rk\"{o}zy \cite{S14}, which we make progress in the current paper.
\begin{conjecture}[S\'{a}rk\"{o}zy]\label{conjecture:S}
If $p$ is a sufficiently large prime and $\lambda \in \F_p^*$, then the shifted subgroup $(S_2-\lambda)\setminus \{0\}$ cannot be written as the product $AB$, where $A,B \subset \F_p^*$ and $|A|, |B| \geq 2$. In other words, $(S_2-\lambda)\setminus \{0\}$ has no non-trivial multiplicative decomposition.
\end{conjecture}

We note that it is necessary to take out the element $0$ from the shifted subgroup, for otherwise one can always decompose $S_2-\lambda$ as $\{0,1\} \cdot (S_2-\lambda)$. It appears that this problem concerning multiplicative decompositions is more difficult than the one concerning additive decompositions stated previously, given that it might depend on the parameter $\lambda$. Inspired by \cref{conjecture:S}, we formulate the following more general conjecture for any proper multiplicative subgroup. For simplicity, we denote $S_d=S_d(\F_q)=\{x^d: x \in \F_q^*\}$ to be the set of $d$-th powers in $\F_q^*$, equivalently, the subgroup of $\F_q^*$ with order $\frac{q-1}{d}$.  
\begin{conjecture}\label{conjecture:MD}
Let $d \geq 2$. If $q \equiv 1 \pmod d$ is a sufficiently large prime power, then for any $\lambda \in \F_q^*$, $(S_d-\lambda)\setminus \{0\}$ does not admit a non-trivial multiplicative decomposition, that is, there do not exist two sets $A,B \subset \F_q^*$ with $|A|,|B| \geq 2$, such that $(S_d-\lambda)\setminus \{0\}=AB$.
\end{conjecture}

\cref{conjecture:MD} predicts that a shifted multiplicative subgroup of a finite field admits a non-trivial multiplicative decomposition only when it has a small size. We remark that the integer version of \cref{conjecture:MD}, namely, for each $k \geq 2$, a non-trivial shift of $k$-th powers in integers has no non-trivial multiplicative decomposition, has been proved and strengthened in a series of papers by Hajdu and S\'{a}rk\"{o}zy \cite{HS18, HS18b, HS20}. On the other hand, to the best knowledge of the authors, \cref{conjecture:MD} appears to be much harder and no partial progress has been made. For the analog of \cref{conjecture:MD} on the additive decomposition of multiplicative subgroups, we refer to recent papers \cite{HP, S14, S20, Y24} for an extensive discussion on partial progress. 

Our main contribution to \cref{conjecture:MD} is the following two results. The first one is a corollary of \cref{thm: stepanovea}.
\begin{corollary}\label{cor:Sidon}
Let $d \geq 2$  and $p$ be a prime such that $d \mid (p-1)$. Let $\lambda \in S_d$. If $(S_d-\lambda) \setminus \{0\}$ can be multiplicative decomposed as the product of two sets $A,B \subset \F_p^*$ with $|A|,|B| \geq 2$, then we must have $|A||B|=|S_d|-1$, that is, all products $ab$ are distinct. In other words, $A,B$ are multiplicatively co-Sidon.    
\end{corollary}

In particular, \cref{cor:Sidon} confirms \cref{conjecture:MD} under the assumption that $q$ is a prime, $\lambda \in S_d$, and $|S_d|-1$ is a prime. The second result provides a partial answer to \cref{conjecture:MD} asymptotically.

\begin{theorem}\label{mainthm5}
Let $d \geq 2$ be fixed and $n$ be a positive integer. As $x \to \infty$, the number of primes $p \leq x$ such that $p \equiv 1 \pmod d$ and $(S_d(\F_p)-n) \setminus \{0\}$ has no non-trivial multiplicative decomposition is at least $$\bigg(\frac{1}{[\mathbb{Q}(e^{2\pi i/d}, n^{1/d}):\mathbb{Q}]}-o(1)\bigg)\pi(x).$$
\end{theorem}

 In particular, by setting $n=1$ and $d=2$, our result has the following significant implication to S\'{a}rk\"{o}zy's conjecture \cite{S14}: for almost all odd primes $p$, the shifted multiplicative subgroup $(S_2(\F_p)-1)\setminus \{0\}$ has no non-trivial multiplicative decomposition. In other words, if $\lambda=1$, then S\'{a}rk\"{o}zy's conjecture holds for almost all primes $p$; see \cref{thm:almostall} for a precise statement. 

Some partial progress has been made for \cref{conjecture:MD} when the multiplicative decomposition is assumed to have special forms \cite{S14,S20}. We also make progress in this direction in~\cref{sec:specialMD}. In particular, in \cref{thm:ternary}, we confirm the ternary version of \cref{conjecture:MD} in a strong sense, which generalizes \cite[Theorem 2]{S14}.

\textbf{Notations.} We follow standard notations from analytic number theory. We use $\pi$ and $\theta$ to denote the standard prime-counting functions. We adopt standard asymptotic notations $O, o, \asymp$. We also follow the Vinogradov notation $\ll$: we write $X \ll Y$ if there is an absolute constant $C>0$ so that $|X| \leq CY$.  

Throughout the paper, let $p$ be a prime and $q$ a power of $p$. Let $\F_q$ be the finite field with $q$ elements and let $\F_q^*=\F_q \setminus \{0\}$. We always assume that $d \mid (q-1)$ with $d \ge 2$, and denote $S_d(\F_q)=\{x^d: x \in \F_q^*\}$ to be the subgroup of $\F_q^*$ with order $\frac{q-1}{d}$. If $q$ is assumed to be fixed, for brevity, we simply write $S_d$ instead of $S_d(\F_q)$.  

We also need some notations for arithmetic operations among sets. Given two sets $A$ and $B$, we write the \emph{product set} $AB=\{ab: a \in A, b \in B\}$, and the \emph{sumset} $A+B=\{a+b: a \in A, b \in B\}$. Given the definition of Diophantine tuples, it is also useful to define the \emph{restricted product set} of $A$, that is, $A \hat{\times} A=\{ab: a,b \in A, a \neq b\}$.

\textbf{Structure of the paper.}
In \cref{sec: Background}, we introduce more background. In \cref{sec: Preliminary estimations}, using Gauss sums and Weil's bound, we give an upper bound on the size of the set which satisfies various multiplicative properties based on character sum estimates. In particular, we prove  \cref{mainprop1}. In \cref{sec: Stepanov}, we first prove \cref{thm: stepanovea} using Stepanov's method. At the end of the section, we deduce applications of \cref{thm: stepanovea} to Diophantine tuples and prove  \cref{mainthm2} and \cref{mainthm3}.
Via Gallagher's larger sieve inequality and other tools from analytic number theory, in \cref{sec: GDM}, we prove 
\cref{mainthm1} and \cref{mainthm1.5}. In \cref{sec: Multiplicative decompositions}, we study multiplicative decompositions and prove \cref{mainthm5}. 

\section{Background}\label{sec: Background}

\subsection{Stepanov's method} 


We first describe Stepanov's method \cite{S69}. If we can construct a low degree \emph{non-zero} auxiliary polynomial that vanishes on each element of a set $A$ with high multiplicity, then we can give an upper bound on $|A|$ based on the degree of the polynomial. \textcolor{black}{It turns out that the most challenging part of our proofs is to show that the auxiliary polynomial constructed is \emph{not identically zero}.}

To check that each root has a high multiplicity, standard derivatives might not work since we are working in a field with characteristic $p$. To resolve this issue, we need the following notation of derivatives, known as 
the \emph{Hasse derivatives} or \emph {hyper-derivatives}; see \cite[Section 6.4]{LN97}.

\begin{definition}
Let $c_0,c_1, \ldots c_d \in \F_q$. If $n$ is a non-negative integer, then the \textit{$n$-th hyper-derivative} of $f(x)=\sum_{j=0}^d c_j x^j$ is
$$
E^{(n)}(f) =\sum_{j=0}^d \binom{j}{n} c_j x^{j-n},
$$
where we follow the standard convention that $\binom{j}{n}=0$ for $j<n$, so that the $n$-th hyper-derivative is a polynomial.
\end{definition}

Following the definition, we have $E^{(0)}f=f$. We also need the next three lemmas.

\begin{lemma}[{\cite[Lemma 6.47]{LN97}}]\label{Leibniz}
If $f, g \in \F_q[x]$, then
$E^{(n)}(fg)= \sum_{k=0}^n E^{(k)} (f) E^{(n-k)} (g).$
\end{lemma}

\begin{lemma}[{\cite[Corollary 6.48]{LN97}}]\label{lem:differentiate}
Let $n,d$ be positive integers. If $a\in \F^{*}_q$ and $c \in \F_q$, then we have
\begin{equation}
E^{(n)}\big((ax+c)^d\big)=a^{n}\binom{d}{n} (ax+c)^{d-n}.
\end{equation}

\end{lemma}

\begin{lemma}[{\cite[Lemma 6.51]{LN97}}]\label{lem:multiplicity}
Let $f$ be a non-zero polynomial in $\F_q[x]$. If $c$ is a root of $E^{(k)}(f)$ for $k=0,1,\ldots, m-1$, then $c$ is a root of multiplicity at least $m$. 
\end{lemma}

\subsection{Gallagher's larger sieve inequality}

In this subsection, we introduce Gallagher's larger sieve inequality and provide necessary estimations from it. Gallagher's larger sieve inequality will be one of the main ingredients for the proof of \cref{mainthm1}. In 1971, Gallagher \cite{G71} discovered the following sieve inequality.
\begin{theorem}\textup{(Gallagher's larger sieve inequality)}\label{larger}  Let $N$ be a natural number and 
 $A\subset\{1,2,\ldots, N\}$.  Let ${\mathcal P}$ be a set of primes.
For each prime $p \in {\mathcal P}$, let $A_p=A \pmod{p}$.
For any $1<Q\leq N$, we have
\begin{equation}
 |A|\leq \frac{\underset{p\leq Q, p\in \mathcal{P}}\sum\log p - \log N}{\underset{p\leq Q, p \in \mathcal{P}}\sum\frac{\log p}{|A_p|}-\log N},
\end{equation}
provided that the denominator is positive.
\end{theorem}

As a preparation to apply Gallagher's larger sieve in our proof, we need to establish a few estimates related to primes in arithmetic progressions. 
For $(a,k)=1$, we follow the standard notation
\[\theta(x;k,a)=\sum_{\substack{p\leq x \\ p\equiv a \text{ mod } k}} \log p.\]
For our purposes, $\log k$ could be as large as $\sqrt{\log \log x}$ (for example, see \cref{mainthm1}), and we need the Siegel-Walfisz theorem to estimate $\theta(x;k,a)$.
\begin{lemma} [{\cite[Corollary 11.21]{MV07}}] \label{Siegel-Walfisz}
Let $A>0$ be a constant. There is a constant $c_1>0$ such that
\[\theta(Q;k,a)=\frac{Q}{\phi(k)} +O_A\bigg(Q\exp(-c_1\sqrt{\log Q})\bigg)\]
holds uniformly for $k\leq (\log Q)^A$ and $(a,k)=1$. 
\end{lemma}

A standard application of partial summation with \cref{Siegel-Walfisz} leads to the following corollary. 

\begin{corollary}\label{partial-summation}
There is a constant $c>0$, such that
$$
\sum_{\substack{p \leq Q\\ p \equiv a \textup{ mod } k}} \frac{\log p}{\sqrt{p}}=\frac{2\sqrt{Q}}{\phi(k)} +O\left( \sqrt{Q} \exp(-c\sqrt{\log Q}) \right)
$$    
holds uniformly for $k \leq \log Q$ and $(a,k)=1$. 
\end{corollary}


We also need the following lemma.

\begin{lemma}[{\cite[page 72]{BM19}}]\label{p mid n}
Let $n$ be a positive integer. Then
$$
\sum_{p \mid n} \frac{\log p}{\sqrt{p}} \ll (\log n)^{1/2}.
$$    
\end{lemma}

\subsection{An effective estimate for $M_k(n, \frac{k}{k-2})${}}
\textcolor{black}{Following \cite{DKM22}, for each real number $L>0$, we write 
\begin{equation*}
    M_k(n;L):=\sup\{|S \cap [n^L,\infty)|: S \text{ satisfies property } D_k(n)\}.
\end{equation*}
It is shown in \cite{DKM22} that $M_k(n,3) \ll_{k} 1$ as $n \to \infty$. For our application, we show a stronger result that $M_k(n,\frac{k}{k-2})\ll_{k} 1$ and we will make this estimate explicit and effective. We follow the proof in \cite{DKM22} closely and prove the following proposition, which will be used later in the proof of \cref{mainthm1}. }

\begin{proposition}\label{prop:effective1}
If $k \geq 3$ and $n>(2ke)^{k^2}$, then $M_k(n,\frac{k}{k-2})\ll \log k \log \log k$, where the implicit constant is absolute.    
\end{proposition}

Let $k \geq 3$. Let $m= M_k(n,\frac{k}{k-2})$ and $A=\{a_1, a_2,  \ldots, a_m\}$ be a generalized $m$-tuple with property $D_k(n)$ and $n^{\frac{k}{k-2}} < a_1<a_2<\cdots <a_m$. Consider the system of equations
\begin{equation}
\label{systemR1}
\left\{ \begin{array}{ll}
a_1 x + n = u^k &\\
a_2 x + n = v^k. &
\end{array} \right.
\end{equation} 
Clearly, for each $i \geq 3$, $x=a_i$ is a solution to this system, and we denote $u_i, v_i$ so that $a_1a_i+n=u_i^k$ and $a_2a_i+n=v_i^k$. 
Let $\alpha:=(a_1/a_2)^{1/k}$.

The following lemma is a generalization of \cite[Lemma 3.1]{DKM22}, showing that $u_i/v_i$ provides a ``good" rational approximation to $\alpha$ if $n$ is large. Note that in \cite[Lemma 3.1]{DKM22}, it was further assumed that $k$ is odd and $L=3$. Nevertheless, an almost identical proof works, and we skip the proof.

\begin{lemma}\label{approximation1}
Let
\begin{equation*}
c(k):= \prod_{j=1}^{\lfloor (k-1)/2\rfloor} \left(\sin \frac{2\pi j}{k}\right)^2.
\end{equation*}
Assume that $n> (2/c(k))^{(k-2)/2}$. Then for each $3 \leq i \leq m$, we have
\begin{equation}
    \bigg|\frac{u_i}{v_i} - \alpha\bigg| \leq \frac{a_2}{2v_i^k}.
\end{equation}
\end{lemma}

\begin{corollary}\label{super}
Assume that $n> (2/c(k))^{(k-2)/2}$. Then $v_i \geq a_2^{4}$ for each $14 \leq i \leq m$ and 
\begin{equation}\label{approx}
\bigg|\frac{u_i}{v_i} - \alpha \bigg| < \frac{1}{v_i^{k-1/2}}.
\end{equation}
\end{corollary}  
\begin{proof}
Let $2 \leq i \leq m-3$. Applying the gap principle from \cite[Lemma 2.4]{DKM22} to $a_i, a_{i+1},a_{i+2}, a_{i+3}$, we have
$$
a_{i+1}a_{i+3} \geq k^kn^{-k} (a_ia_{i+2})^{k-1} \geq k^k n^{-k} (a_ia_{i+1})^{k-1}. 
$$
It follows that $$a_{i+3} \geq  a_i^{k-1}a_{i+1}^{k-2}n^{-k} \geq a_i^{k-1}.$$ 
In particular $a_{14} \geq a_2^{(k-1)^4}\geq a_2^{4k}$. Thus, if $i \geq 14$, then $v_i \geq a_i^{1/k} \geq a_{14}^{1/k} \geq a_2^4$ and inequality~\eqref{approx} follows from \cref{approximation1}.
\end{proof}

Now we are ready to prove \cref{prop:effective1}. Recall we have the inequality $\sin x \geq \frac{2x}{\pi}$ for $x \in [0,\frac{\pi}{2}]$, and the inequality
$s! \geq (s/e)^s $
for all positive integers $s$. It follows that
$$
\sqrt{c(k)}=\prod_{j=1}^{(k-1)/2} \sin \frac{2\pi j}{k}=\prod_{j=1}^{(k-1)/2} \sin \frac{\pi j}{k} \geq \prod_{j=1}^{(k-1)/2} \frac{2j}{k}=\frac{((k-1)/2)!}{(k/2)^{(k-1)/2}} \geq  \bigg(\frac{k-1}{ke}\bigg)^{(k-1)/2}
$$
when $k$ is odd, and 
$$
(c(k))^{1/4}=\sqrt{\prod_{j=1}^{(k-2)/2} \sin \frac{2\pi j}{k}}=\prod_{j=1}^{\lfloor k/4 \rfloor} \sin \frac{\pi j}{k/2} \geq \prod_{j=1}^{\lfloor k/4 \rfloor} \frac{4j}{k}=\frac{(\lfloor k/4 \rfloor)!}{(k/4)^{\lfloor k/4 \rfloor}} \geq  \bigg(\frac{4\lfloor k/4 \rfloor}{ke}\bigg)^{\lfloor k/4 \rfloor}
$$
when $k$ is even. Thus, when $k \geq 3$, we always have
$$
\frac{2}{c(k)}\leq 2\bigg(\frac{ke}{k-2}\bigg)^{k} \leq 2(ke)^k.
$$
Therefore, when $n>(2ke)^{k^2}$, we can apply \cref{approximation1}. Note that the absolute height of $\alpha$ is  $H(\alpha) \leq a_2^{1/k}$.  Since $k \geq 3$, for $14\leq i \leq m$, \cref{approximation1} implies that
\begin{equation*}
    \left| \frac{u_i}{v_i} - \alpha\right| \leq \frac{1}{v_i^{k-1/2}} \leq   \frac{1}{v_i^{2.5}};
\end{equation*}
moreover, $\max(u_i,v_i) = v_i>a_2^{1/k} \geq \max(H(\alpha), 2)$. Therefore, we can apply the quantitative Roth's theorem due to Evertse \cite{E10} (see also \cite[Theorem 2.2]{DKM22}) to conclude that
$$
m \leq 13+ 2^{28} \log (2k) \log (2\log 2k) \ll \log k \log \log k,
$$
where the implicit constant is absolute.

\subsection{Implications of the Paley graph conjecture}\label{sec:PaleyABC}

The Paley graph conjecture on double character sums implies many results of the present paper related to the estimation of character sums. We record the statement of the conjecture (see for example \cite{DKM22,GM20}).

\begin{conjecture}[Paley graph conjecture]\label{Paley-graph-conjecture}
Let $\epsilon >0$ be a real number. Then there is $p_0=p_0(\epsilon)$ and $\delta=\delta(\epsilon)>0$ such that for any prime $p>p_0$, any $A, B \subseteq \mathbb{F}_p$ with $|A|,|B| > p^{\epsilon}$, and any non-trivial multiplicative character $\chi$  of $\F_p$, the following inequality holds:
\begin{equation*}
    \bigg| \sum_{a\in A,\,  b\in B} \chi(a+b)\bigg| \leq p^{-\delta} |A| |B|.
\end{equation*}
\end{conjecture}

The connection between the Paley graph conjecture and the problem of bounding the size of Diophantine tuples was first observed by G\"{u}lo\u{g}lu and Murty in \cite{GM20}. Let $d \geq 2$ be fixed, $\lambda \in \F_p^*$, where $p \equiv 1 \pmod d$ is a prime. The Paley graph conjecture trivially implies $MD_{d}(\lambda, \F_p)=p^{o(1)}$ and $MSD_{d}(\lambda, \F_p)=p^{o(1)}$ as $p \to \infty$. Also, the bound on $M_k(n)$ in \cref{mainthm1} can be improved to $(\log n)^{o(1)}$ (see \cite{DKM22}, \cite{GM20}) when $k$ is fixed and $n \to \infty$. Furthermore, the Paley graph conjecture also immediately implies S\'{a}rk\"{o}zy's conjecture (\cref{conjecture:S}) in view of \cref{thm:minlb}. However, the Paley graph conjecture itself remains widely open, and our results are unconditional. We refer to \cite{SS22} and the references therein for recent progress on the Paley graph conjecture assuming $A,B$ have small doubling.

\section{Preliminary estimations for product sets in shifted multiplicative subgroups}\label{sec: Preliminary estimations}

\subsection{Character sum estimate and the square root upper bound}

The purpose of this subsection is to prove \cref{mainprop1} by establishing an upper bound on the double character sum in \cref{charsum:sym} using basic properties of characters and Gauss sums. For any prime $p$, and any $x \in \F_p$, we follow the standard notation that $e_p(x)=\exp(2 \pi i x/p)$, where we embed $\F_p$ into $\Z$. We refer the reader to \cite[Chapter 5]{LN97} for more results related to Gauss sums and character sums. 

We refer to \cite[Section 2]{BM19} for a historical discussion of Vinogradov's inequality~\cref{Vinogradov}. Gyarmati \cite[Theorem 7]{G01} and Becker and Murty \cite[Proposition 2.7]{BM19} independently showed that the Legendre symbol in inequality~\cref{Vinogradov} can be replaced with any non-trivial Dirichlet character modulo $p$. For our purposes, we extend Vinogradov's inequality \cref{Vinogradov} to all finite fields $\F_q$ and all nontrivial multiplicative characters $\chi$ of $\F_q$, with a slightly improved upper bound. 

\begin{proposition}\label{charsum:sym}
Let $\chi$ be a non-trivial multiplicative character of $\F_q$ and $\lambda \in \F_q^*$. For any $A,B \subset \F_q^*$,  we have
  $$
 \bigg|\sum_{a\in A,\, b\in B}\chi(ab+\lambda)\bigg|  \leq \sqrt{q|A||B|}\bigg(1-\frac{\max\{|A|,|B|\}}{q}\bigg)^{1/2}.
 $$    
\end{proposition}

Before proving \cref{charsum:sym}, we need some preliminary estimates. Let $\chi$ be a multiplicative character of $\F_q$; then the Gauss sum associated to $\chi$ is defined to be $$G(\chi)=\sum_{ c \in \F_q} \chi(c) e_p\big(\Tr_{\F_q}(c)\big),$$
where $\Tr_{\F_q}:\F_q \to \F_p$ is the absolute trace map.

\begin{lemma}[{\cite[Theorem 5.12]{LN97}}] \label{qf}
Let $\chi$ be a multiplicative character of  $\F_q$. Then for any $a \in \F_q$, 
$$
\overline{\chi(a)}=\frac{1}{G(\chi)} \sum_{ c \in \F_q} \chi(c) e_p\big(\Tr_{\F_q}(ac)\big). 
$$
\end{lemma}



Now we are ready to prove \cref{charsum:sym}.

\begin{proof}[Proof of \textup{\cref{charsum:sym}}]

By \cref{qf}, we can write

\begin{align*}
\sum_{a\in A,\, b\in B}\overline{\chi(ab+\lambda)}
&=\sum_{a\in A,\, b\in B}\overline{\chi(b)} \overline{\chi(a+\lambda b^{-1})}\\
&=\frac{1}{G(\chi)} \sum_{c \in \F_q} \chi(c) \sum_{a\in A,\, b\in B} \overline{\chi(b)} e_p\big(\Tr((a+\lambda b^{-1})c)\big).
\end{align*}

It is well-known that $|G(\chi)|= \sqrt{q}$ (see for example \cite[Theorem 5.11]{LN97}). Since $|\chi(c)|=1$ for each $c \in \F_q^*$, we can apply the triangle inequality and Cauchy-Schwarz inequality to obtain
\begin{align*}
\bigg|\sum_{a\in A,\, b\in B}\chi(ab+\lambda)\bigg| 
& \leq \frac{1}{\sqrt{q}}  \sum_{c \in \F_q^*} \bigg|\sum_{a\in A,\, b\in B} \overline{\chi(b)} e_p\big(\Tr((a+\lambda b^{-1})c)\big)\bigg|\\
&\leq \frac{1}{\sqrt{q}}  \bigg(\sum_{c \in \F_q^*} \bigg|\sum_{a\in A} e_p\big(\Tr(ac)\big)\bigg|^2\bigg)^{1/2} \bigg(\sum_{c \in \F_q^*}\bigg|\sum_{b\in B} \overline{\chi(b)}e_p\big(\Tr(b^{-1}c)\big)\bigg|^2\bigg)^{1/2}.\end{align*}
By orthogonality relations, we have
$$
\sum_{c \in \F_q^*} \bigg|\sum_{a\in A} e_p\big(\Tr(ac)\big)\bigg|^2=q|A|-|A|^2, \quad
\sum_{c \in \F_q^*}\bigg|\sum_{b\in B} \overline{\chi(b)}e_p\big(\Tr(b^{-1}c)\big)\bigg|^2\leq q|B|.
$$
Thus, we have 
 $$
 \bigg|\sum_{a\in A,\, b\in B}\chi(ab+\lambda)\bigg|  \leq \sqrt{q|A||B|}\bigg(1-\frac{|A|}{q}\bigg)^{1/2}.
 $$
By switching the roles of $A$ and $B$, we obtain the required character sum estimate.
\end{proof}

Let $d \mid (q-1)$ such that $d \ge 2$ and denote $S_{d}=\{x^d \colon x \in \mathbb{F}_{q}^{*}\}$ with order $\frac{q-1}{d}$.
Then we prove \cref{mainprop1}.

\begin{proof} [Proof of \textup{\cref{mainprop1}}]
Let $\chi$ be a multiplicative character of order $d$. 

Let $A \subset \F_q^*$ with property $SD_d(\lambda, \F_q)$, that is, $AA+\lambda \subset S_d \cup \{0\}$. Note that $\chi(ab+\lambda)=1$ for each $a,b \in A$, unless $ab+\lambda=0$. 
Note that given $a \in A$, there is at most one $b \in A$ such that $ab+\lambda=0$. Therefore, by \cref{charsum:sym}, we have
$$
|A|^2-|A|\leq  \bigg|\sum_{a,b\in A}\chi(ab+\lambda)\bigg|\leq \sqrt{q} |A| \bigg(1-\frac{|A|}{q}\bigg)^{1/2}.
$$
It follows that
$$
(|A|-1)^2 \leq q-|A| \implies |A|\leq \frac{\sqrt{4q-3}+1}{2}.
$$

Next we work under the weaker assumption $A \hat{\times} A+\lambda \subset S_d \cup \{0\}$. In this case, note that $\chi(ab+\lambda)=1$ for each $a,b \in A$ such that $a\neq b$, unless $ab+\lambda=0$. \cref{charsum:sym} then implies that
$$
|A|^2-3|A|\leq  \bigg|\sum_{a,b\in A}\chi(ab+\lambda)\bigg|\leq \sqrt{q} |A| \bigg(1-\frac{|A|}{q}\bigg)^{1/2}
$$
and it follows that $|A|\leq \sqrt{q-\frac{11}{4}}+\frac{5}{2}$.
\end{proof}

\subsection{Estimates on \texorpdfstring{$|A|$}{} and \texorpdfstring{$|B|$}{} if \texorpdfstring{$AB=(S_d-\lambda)\setminus \{0\}$}{}}

Let $A,B \subset \F_q^*$ and $\lambda \in \F_q^*$.
In this subsection, we provide several useful estimates on $|A|$ and $|B|$ when $AB=(S_{d}-\lambda) \setminus \{0\}$, which will be used in \cref{sec: Multiplicative decompositions}.

We need to use the following lemma, due to Karatsuba \cite{K91}.


\begin{lemma}\label{prop:asymmetric}
 Let $A,B \subset \F_q^*$ and $\lambda \in \F_q^*$. Then for any non-trivial multiplicative character $\chi$ of $\F_q$ and any positive integer $\nu$, we have
$$
\sum_{\substack{a\in A\\ b\in B}} \chi(ab+\lambda) \ll_{\nu} |A|^{(2\nu-1)/2\nu} (|B|^{1/2}q^{1/2\nu}+|B|q^{1/4\nu}).
$$   
\end{lemma}



The following proposition improves and generalizes \cite[Theorem 1]{S14}. It also improves \cite[Lemma 17]{S20} (see \cref{remark:logp}).

\begin{proposition}\label{thm:minlb}
Let $\epsilon>0$. Let $d \mid (q-1)$ such that $2 \leq d \leq q^{1/2-\epsilon}$ and $\lambda \in \F^*_q$.  If $AB=(S_d-\lambda)\setminus \{0\}$ for some $A,B \subset \F_q^*$ with $|A|, |B| \geq 2$, then
    $$
    \frac{\sqrt{q}}{d} \ll \min \{|A|,|B|\} \leq \max \{|A|,|B|\} \ll q^{1/2}.
    $$    
\end{proposition}
\begin{proof}
Let $A, B \subset \F_q^*$ and $\lambda \in \F_q^*$ such that
$AB=(S_d-\lambda)\setminus \{0\}$ with $|A|, |B|\geq 2$. Without loss of generality, we assume that $|A|\geq |B|$. We first establish a weaker lower bound that $|B|\gg q^{\epsilon/2}$.

When $d=2$, S\'{a}rk\"{o}zy \cite{S14} has shown that $|B|\gg \frac{\sqrt{q}}{3\log q}$. While he only proved this estimate when $q=p$ is a prime \cite[Theorem 1]{S14}, it is clear that the same proof extends to all finite fields $\F_q$. 

Next, assume that $d \geq 3$. Let $B=\{b_1, b_2, \ldots, b_k\}$ and $|B|=k$. Since $AB \subset S_d-\lambda$, we have $AB+\lambda \subset S_d$. Let $\chi$ be a multiplicative character of $\F_q$ with order $d$. Then it follows that for each $a \in A$, we have $\chi(a+\lambda b_i^{-1})=1/\chi(b_i)$ for each $1 \leq i \leq k$. 
Therefore, by a well-known consequence of Weil's bound (see for example \cite[Exercise 5.66]{LN97}),
$$
|A|\leq \frac{q}{d^k}+\bigg(k-1-\frac{k}{d}+\frac{1}{d^k}\bigg)\sqrt{q}+\frac{k}{d}< \frac{q}{d^k}+k\sqrt{q}.
$$
On the other hand, since $AB=(S_d-\lambda)\setminus \{0\}$, we have
$$
|A||B|\geq |AB| \geq |S_d|-1=\frac{q-1}{d}-1.
$$
Combining the above two inequalities, we obtain that
$$
\frac{2q}{d^2}+k^2\sqrt{q}\geq \frac{kq}{d^k}+k^2\sqrt{q}>|A||B| \geq \frac{q-1}{d}-1.
$$
Since $d \geq 3$, it follows that $k^2\sqrt{q}\gg \frac{q}{d}$ and thus
$$
|B|=k \gg \frac{q^{1/4}}{\sqrt{d}}\gg q^{\epsilon/2}.
$$

Let $\nu=\lceil 2/\epsilon \rceil$. By \cref{prop:asymmetric}, and as  $|B| \gg q^{1/\nu}$, we have
$$
|A||B|=\sum_{\substack{a\in A\\ b\in B}} \chi(ab+\lambda) \ll |A|^{(2\nu-1)/2\nu} (|B|^{1/2}q^{1/2\nu}+|B|q^{1/4\nu})\ll  |A|^{(2\nu-1)/2\nu}|B|q^{1/4\nu}.
$$
It follows that $|A|\ll q^{1/2}$. Thus, $|B|\gg |S_d|/|A|\gg q^{1/2}/d$.
\end{proof}

\begin{remark}
We remark that the same method could be used to refine a similar result for the additive decomposition of multiplicative subgroups, which improves a result of Shparlinski \cite[Theorem 6.1]{S13} (moreover, our proof appears to be much simpler than his proof). More precisely, we can prove the following:

Let $\epsilon>0$. Let $d \mid (q-1)$ such that $2 \leq d \leq q^{1/2-\epsilon}$.  If $A+B=S_d$ for some $A,B \subset \F_q$ with $|A|, |B| \geq 2$, then
    $$
        \frac{\sqrt{q}}{d} \ll \min \{|A|,|B|\} \leq \max \{|A|,|B|\} \ll q^{1/2}.
    $$   
\end{remark}

Note that \cref{thm:minlb} only applies to multiplicative subgroups $G=S_d$ with $|G|>\sqrt{q}$. When $q=p$ is a prime, and $G$ is a non-trivial multiplicative subgroup of $\F_p$ (in particular, $|G|<\sqrt{p}$ is allowed), we have the following estimate, due to Shkredov \cite{S20}.

\begin{lemma}\label{1/2+o(1)}
If $G$ is a proper multiplicative subgroup of $\F_p$ such that $AB=(G-\lambda)\setminus \{0\}$ for some $A,B \subset \F_p^*$ with $|A|,|B| \geq 2$ and some $\lambda \in \F_p^*$, then 
$$
|G|^{1/2 + o(1)}= \min\{|A|, |B|\} \leq \max \{|A|, |B|\}=|G|^{1/2 + o(1)}
$$
as $|G| \to \infty$.
\end{lemma}  

\begin{proof}
If $0 \in G-\lambda$, let $A'=A \cup \{0\}$; otherwise, let $A'=A$. Then we have $A'B=G-\lambda$, or equivalently, $A'/(B^{-1})=G-\lambda$. Note that $|A' \setminus \{0\}|=|A|\geq 2$ and $|B^{-1}|=|B| \geq 2$, thus, the lemma then follows immediately from \cite[Lemma 17]{S20}. 
\end{proof}

\begin{remark}\label{remark:logp}
Let $AB=(S_d-\lambda) \setminus \{0\}$, where  $A,B \subset \F_p^*$ with $|A|, |B| \geq 2$ and $\lambda \in \F_p^*$. When $d$ is a constant, and $p \to \infty$, \cref{thm:minlb} is better than \cref{1/2+o(1)}. Indeed, in \cite[Lemma 17]{S20}, Shkredov showed that $\max\{|A|, |B|\} \ll \sqrt{p}\log p$, while \cref{thm:minlb} showed the stronger result that $\max\{|A|, |B|\} \ll \sqrt{p}$, removing the $\log p$ factor. This stronger bound would be crucial in proving results related to multiplicative decompositions (for example, \cref{thm:almostall}). Instead, \cref{1/2+o(1)} will be useful for applications in ternary decompositions (\cref{thm:ternary}).
\end{remark}

\begin{remark}\label{remark:1/2+o(1)}
\cref{1/2+o(1)} fails to extend to $\F_q$, where $q$ is a proper prime power. Let $q=r^2$ be a square, $G=\F_r^*$, and $\lambda=-1$. Let $A$ be a subset of $\F_r^*$ with size $\lfloor (r-1)/2 \rfloor$. Let $B=\F_r^* \setminus A^{-1}$. Then we have $AB=\F_r^* \setminus \{1\}=(G+1)\setminus \{0\}$ while $|A|, |B| \gg |G|$. 
\end{remark}

\begin{remark}
Let $d \geq 2$ be fixed. Let $q \equiv 1 \pmod d$ be a prime power and let $\lambda \in \F_{q}^*$, we define $N(q,\lambda)$ be the total number of pairs $(A,B)$ of sets $A,B \subseteq \mathbb{F}_{q}$ with $|A|, |B| \geq 2$ such that $AB = (S_{d} - \lambda) \setminus \{0 \}$. Note that \cref{conjecture:MD} implies $N(q,\lambda)=0$ when $q$ is sufficiently large, but it seems out of reach in general. Instead, one can find a non-trivial upper bound of $N(q,\lambda)$. Using \cref{thm:minlb} and following the same strategy of the proof in \cite[Theorem 1]{BKS15},  we have a non-trivial upper bound of $N(q,\lambda)$ as follows:
\[N(q, \lambda) \le \exp \left( O(q^{1/2}) \right).\]

We also note that by using \cref{cor:Sidon}, we confirmed $N(q, \lambda)=0$ if $q$ is a prime, $\lambda \in S_{d}$, and $|S_d|-1$ is a prime. 
In addition, \cref{mainthm5} gives the same result for $N(p,n)$ asymptotically for a subset of primes $p$ with lower density at least $\frac{1}{[\mathbb{Q}(e^{2\pi i/d}, n^{1/d}):\mathbb{Q}]}$.
\end{remark}

\section{Products and restricted products in shifted multiplicative subgroups}\label{sec: Stepanov}

 In this section, we use Stepanov's method to study product sets and restricted product sets that are contained in shifted multiplicative subgroups. Our proofs are inspired by Stepanov's original paper \cite{S69}, and the recent breakthrough of Hanson and Pertidis \cite{HP}, together with its extensions and applications developed by the second author \cite{Y24, Y25}. 

 Throughout the section, we assume $d \geq 2$ and $q \equiv 1 \pmod d$ is a prime power. Recall that $S_d=S_d(\F_q)=\{x^d: x \in \F_q^*\}$.

\subsection{Product set in a shifted multiplicative subgroup}

In this subsection, we prove \cref{thm: stepanovea}, which can be viewed as the bipartite version of Diophantine tuples over finite fields.
As a corollary of \cref{thm: stepanovea}, we prove \cref{cor:Sidon}. 
Besides it, \cref{thm: stepanovea} will be also repeatedly used to prove several of our main results in the present paper.

\begin{proof}[Proof of \textup{\cref{thm: stepanovea}}]

Let $r=|B \cap (-\lambda A^{-1})|$.
Let $A=\{a_1,a_2,\ldots, a_n\}$ and $B=\{b_1,b_2, \ldots, b_m\}$ such that $b_{r+1}, \ldots, b_m \notin (-\lambda A^{-1})$. Since $AB+\lambda \subset S_d \cup \{0\}$, we have $$(a_ib_j+\lambda)^{\frac{q-1}{d}+1}=a_ib_j+\lambda$$ for each $1 \leq i \leq n$ and $1 \leq j \leq m$. This simple observation will be used repeatedly in the following computation.

Let $c_1,c_2,...,c_n \in \F_q$ be the unique solution of the following system of equations:
\begin{equation} \label{system} 
\left\{
\TABbinary\tabbedCenterstack[l]{
\sum_{i=1}^n c_i=1\\\\
\sum_{i=1}^n c_i a_i^j=0,  \quad 1 \leq j \leq n-1
} .\right.    
\end{equation}
This is justified by the invertibility of the coefficient matrix of the system (a Vandermonde matrix). We claim that $\sum_{i=1}^n c_ia_i^n\neq 0$. Suppose otherwise that $\sum_{i=1}^n c_ia_i^n=0$, then $c_i=0$ for all $i$, violating the assumption $\sum_{i=1}^n c_i=1$ in equation~\eqref{system}. Indeed, the generalized Vandermonde matrix $(a_i^j)_{1 \leq i \leq n, 1 \leq j \leq n}$ is non-singular since it has determinant $$a_1 a_2 \ldots a_n \prod_{i<j} (a_j-a_i)\neq 0.$$

Consider the following auxiliary polynomial 

\begin{equation}
\label{auxiliary1}
f(x)=-\lambda^{n-1}+\sum_{i=1}^n c_i (a_ix+\lambda)^{n-1+\frac{q-1}{d}}\in \F_q[x].    
\end{equation}

Note that $n=|A| \leq |S_d \cup \{0\}|\leq \frac{q-1}{d}+1$. Thus, if $q=p$ is a prime, then $n-1+\frac{p-1}{d} \leq \frac{2(p-1)}{d} \leq p-1$ and thus the condition $\binom{n-1+\frac{q-1}{d}}{n} \not \equiv 0 \pmod p$ is automatically satisfied. 
Then $f$ is a non-zero polynomial since the coefficient of $x^n$ in $f$ is 
$$
\binom{n-1+\frac{q-1}{d}}{n} \cdot \lambda^{\frac{q-1}{d}-1} \cdot \sum_{i=1}^n c_i a_i^n \neq 0
$$
by the assumption on the binomial coefficient. Also, it is clear that the degree of $f$ is at most $n-1+\frac{q-1}{d}$.

Next, we compute the derivatives of $f$ on $B$. For each $1\leq j \leq m$, system \cref{system} implies that
\begin{align*}
E^{(0)} f (b_j)
=  -\lambda^{n-1}+\sum_{i=1}^n c_i (a_ib_j+\lambda)^{n-1}
=  -\lambda^{n-1}+\sum_{\ell=0}^{n-1} \binom{n-1}{\ell} \lambda^{n-1-\ell}\bigg(\sum_{i=1}^n c_i a_i^\ell\bigg)b_j^{\ell}  
=0.
\end{align*}

For each $1\leq j \leq m$ and $1 \leq k \leq n-2$, we have that 
\begin{align*}
E^{(k)} f (b_j)
&= \binom{n-1+\frac{q-1}{d}}{k}  \sum_{i=1}^n c_i a_i^k (a_ib_j+\lambda)^{n-1+\frac{q-1}{d}-k} \\
&= \binom{n-1+\frac{q-1}{d}}{k}  \sum_{i=1}^n c_i a_i^k (a_ib_j+\lambda)^{n-1-k}  \\
&= \binom{n-1+\frac{q-1}{d}}{k}  \sum_{\ell=0}^{n-1-k} \binom{n-1-k}{\ell} \lambda^{n-1-k-\ell} \bigg(\sum_{i=1}^n c_i a_i^{k+\ell}\bigg)b_j^{\ell}   
=0,
\end{align*}
where we use \cref{lem:differentiate} and the assumptions in system \cref{system}.    

For each $r+1\leq j \leq m$, by the assumption, $b_j \notin (-\lambda A^{-1})$, that is, $a_ib_j+\lambda \neq 0$ for each $1 \leq i \leq n$. Thus, for each $r+1\leq j \leq m$, we additionally have
\begin{align*}
E^{(n-1)} f (b_j)
&= \binom{n-1+\frac{q-1}{d}}{n-1}  \sum_{i=1}^n c_i a_i^{n-1}(a_ib_j+\lambda)^{\frac{q-1}{d}} = \binom{n-1+\frac{q-1}{d}}{n-1}  \sum_{i=1}^n c_i a_i^{n-1}=0.
\end{align*}

Therefore, \cref{lem:multiplicity} allows us to conclude that each of $b_1,b_2, \ldots b_r$ is a root of $f$ with multiplicity at least $n-1$, and each of $b_{r+1},b_{r+2}, \ldots b_m$ is a root of $f$ with multiplicity at least $n$. 
It follows that
$$
r(n-1)+(m-r)n= mn-r \leq \operatorname{deg}f \leq \frac{q-1}{d}+n-1.
$$

Finally, assuming that $\lambda \in S_d$. In this case, 
$$
f(0)=-\lambda^{n-1}+\lambda^{n-1+\frac{q-1}{d}}\sum_{i=1}^n c_i=-\lambda^{n-1}+\lambda^{n-1}=0.
$$
And the coefficient of $x^j$ of $f$ is $0$ for each $1 \leq j \leq n-1$ by the assumptions on $c_i$'s. It follows that $0$ is also a root of $f$ with multiplicity $n$. Since $0 \notin B$, we have the stronger estimate that
$mn-r+n\leq \frac{q-1}{d}+n-1.$

\end{proof}
\begin{remark}
More generally, one can study the same question if $AB+\lambda$ is instead contained in a coset of $S_d$. However, note that this more general case can be always reduced to the special case studied in \cref{thm: stepanovea}. Indeed, if $AB+\lambda \subset \xi S_d \cup \{0\}$ with $\xi \in \F_q^*$, then $A'B+\lambda/\xi \subset S_d \cup \{0\}$, where $A'=A/\xi$. 
\end{remark}

Next, we prove \cref{cor:Sidon}, an important corollary of \cref{thm: stepanovea}. It would be crucial for proving results in \cref{sec: Multiplicative decompositions}.

\begin{proof}[Proof of \textup{\cref{cor:Sidon}}]
Since $0 \notin AB+\lambda$, we have $B \cap (-\lambda A^{-1})=\emptyset$, thus \cref{thm: stepanovea} implies that $|A||B| \leq |S_d|-1$. On the other hand, since $(S_d-\lambda) \setminus \{0\}=AB$, it follows that $|A||B| \geq |AB|=|(S_d-\lambda) \setminus \{0\}|=|S_d|-1$. Therefore, we have $|A||B|=|AB|=|S_d|-1$.
\end{proof}

\subsection{Restricted product set in a shifted multiplicative subgroup}

Recall that $A' \subset \F_q^*$ has property $D_d(\lambda, \F_q)$ if and only if $ab+\lambda \in S_d \cup \{0\}$ for each $a,b \in A'$ such that $a \neq b$. In other words, $A' \hat{\times} A'+\lambda \subset S_d \cup \{0\}$. Thus, in this subsection, we are led to study restricted product sets and we establish the following restricted product analog of \cref{thm: stepanovea}. In the next subsection, \cref{thm: stepanovea} and \cref{thm: restricted} will be applied together to prove \cref{mainthm2} in the case that $q$ is a prime and \cref{mainthm3} in the case that $q$ is a square. 

\begin{theorem}\label{thm: restricted}
Let $d \geq 2$ and let $q \equiv 1 \pmod d$ be a prime power. Let $A' \subset \F_q^*$ and $\lambda \in \F_q^*$. If $A' \hat{\times} A'+\lambda \subset S_d \cup \{0\}$ while $A'A'+\lambda \not \subset S_d \cup \{0\}$, then $|A'|\leq \sqrt{2(q-1)/d}+4$. 
\end{theorem}

The proof of \cref{thm: restricted} is similar to \cref{thm: stepanovea}, but it is more delicate. In particular, the choice of the auxiliary polynomial \cref{auxiliary2} needs to be modified from that of the proof \cref{auxiliary1} of \cref{thm: stepanovea}. In view of \cref{thm: stepanovea}, we can further assume that $A'A'+\lambda \not \subset S_d \cup \{0\}$, for otherwise we already have a good bound on $|A'|$; we refer to \cref{Subsec: applications} for details. It turns out that this additional assumption (which we get for free) is crucial in our proof since it guarantees that the auxiliary polynomial we constructed is not identically zero. 
\begin{proof}[Proof of \textup{\cref{thm: restricted}}]
Since $A'A'+\lambda \not \subset S_d \cup \{0\}$, there is $b \in A'$ such that  $b^2+\lambda \not \subset S_d \cup \{0\}$. Let $A''=A' \setminus \{-\lambda/b\}$. If $|A''|=1$, then we are done. Otherwise, if $|A''|$ is even, let $A=A''$; if $|A''|$ is odd, let $A=A'' \setminus \{b'\}$, where $b' \in A''$ is an arbitrary element such that $b' \neq b$. Then we have $b \in A$ and $|A|$ is even. Note that $|A| \geq |A'|-2$, thus it suffices to show $|A|\leq \sqrt{2(q-1)/d}+2$.

Let $|A|=n$, where $n$ is even. Write $A=\{a_1, a_2, \ldots, a_n\}$. Without loss of generality, we may assume that $a_1=b$. Let $m=n/2-1$. 
Let $c_1,c_2,...,c_n \in \F_q$ be the unique solution of the following system of equations:
\begin{equation} \label{system2} 
\left\{
\TABbinary\tabbedCenterstack[l]{
\sum_{i=1}^n c_i a_i^j=0, \quad -m \leq j \leq m\\\\
\sum_{i=1}^n c_i a_i^{m+1}=1.
}\right.    
\end{equation}
Indeed, that coefficient matrix of the system is the generalized Vandermonde matrix $(a_i^j)_{1 \leq i \leq n, -m \leq j \leq m+1},$
which is non-singular since it has nonzero determinant $(a_1 a_2 \ldots a_n)^{-m} \prod_{i<j} (a_j-a_i)\neq 0.$
Note that $c_1 \neq 0$; for otherwise $c_1=0$ and we must have $c_1=c_2=\ldots=c_n=0$ in view of the first $n-1$ equations in system \cref{system2}, which contradicts the last equation in system \cref{system2}.

Consider the following auxiliary polynomial 
\begin{equation}
\label{auxiliary2}
f(x)=\sum_{i=1}^n c_i (a_ix+\lambda)^{m+\frac{q-1}{d}} (a_i^{-1}x-1)^m\in \F_q[x].
\end{equation}
It is clear that the degree of $f$ is at most $2m+\frac{q-1}{d}$. Since $A \hat{\times} A+\lambda \subset S_d \cup \{0\}$, we have $$(a_ia_j+\lambda)^{\frac{q-1}{d}+1} (a_i^{-1} a_j-1)=(a_ia_j+\lambda)(a_i^{-1} a_j-1)$$ 
for each $1 \leq i,j \leq n$. 
This simple observation will be used repeatedly in the following computation.

First, we claim that for each $0 \leq k_1<m$, $0\leq k_2<m$, and $1 \leq j \leq n$, we have
\begin{equation}\label{eq:k1k2}
\sum_{i=1}^n c_i E^{(k_1)}[(a_ix+\lambda)^{m+\frac{q-1}{d}}](a_j) \cdot E^{(k_2)}[(a_i^{-1}x-1)^m](a_j)=0.
\end{equation}
Indeed, by \cref{lem:differentiate}, we have
\begin{align*}
&\sum_{i=1}^n c_i E^{(k_1)}[(a_ix+\lambda)^{m+\frac{q-1}{d}}](a_j) \cdot E^{(k_2)}[(a_i^{-1}x-1)^m](a_j)\\
&= \binom{m+\frac{q-1}{d}}{k_1} \binom{m}{k_2} \bigg(\sum_{i=1}^n c_i a_i^{k_1-k_2} (a_ja_i+\lambda)^{m-k_1} (a_i^{-1}a_j-1)^{m-k_2}\bigg) \\
&= \binom{m+\frac{q-1}{d}}{k_1} \binom{m}{k_2} \sum_{\ell_1=0}^{m-k_1} \sum_{\ell_2=0}^{m-k_2}  \binom{m-k_1}{\ell_1} \binom{m-k_2}{\ell_2}
\bigg(\sum_{i=1}^n c_i a_i^{k_1-k_2} (a_ja_i)^{\ell_1}  \lambda^{m-k_1-\ell_1} (a_i^{-1}a_j)^{\ell_2}(-1)^{m-k_2-\ell_2}\bigg)\\
&= \binom{m+\frac{q-1}{d}}{k_1} \binom{m}{k_2} \sum_{\ell_1=0}^{m-k_1} \sum_{\ell_2=0}^{m-k_2}  \binom{m-k_1}{\ell_1} \binom{m-k_2}{\ell_2} a_j^{\ell_1+\ell_2} \lambda^{m-k_1-\ell_1} (-1)^{m-k_2-\ell_2}
\bigg(\sum_{i=1}^n c_i a_i^{(k_1+\ell_1)-(k_2+\ell_2)} \bigg).
\end{align*}
Note that in the exponent of the last summand, we always have $0 \leq k_1+\ell_1 \leq m$ and $0 \leq k_2+\ell_2 \leq m$ so that $-m \leq (k_1+\ell_1)-(k_2+\ell_2) \leq m$, and thus 
$$
\sum_{i=1}^n c_i a_i^{(k_1+\ell_1)-(k_2+\ell_2)}=0
$$
by the assumptions in system \cref{system2}. This proves the claim. 

Then, for each $1\leq j \leq n$ and $0 \leq r \leq m-1$, we apply \cref{Leibniz} and equation \cref{eq:k1k2} in the above claim to obtain that 
\begin{align*}
E^{(r)} f (a_j)
=  \sum_{i=1}^n c_i \bigg(\sum_{k=0}^r E^{(k)}[(a_ix+\lambda)^{m+\frac{q-1}{d}}](a_j) \cdot E^{(r-k)}[(a_i^{-1}x-1)^m](a_j)\bigg)=0.
\end{align*}

Similarly, using \cref{Leibniz}, \cref{lem:differentiate},  system \cref{system2}, and equation \cref{eq:k1k2}, we can compute
\begin{align*}
E^{(m)} f (a_1)
&=  \sum_{i=1}^n c_i \bigg(\sum_{k=0}^m E^{(k)}[(a_ix+\lambda)^{m+\frac{q-1}{d}}](a_1) \cdot E^{(m-k)}[(a_i^{-1}x-1)^m](a_1)\bigg)\\
&= \sum_{i=1}^n c_i \bigg(E^{(0)}[(a_ix+\lambda)^{m+\frac{q-1}{d}}](a_1) \cdot E^{(m)}[(a_i^{-1}x-1)^m](a_1)\bigg) \\
& \quad + \sum_{i=1}^n c_i \bigg(E^{(m)}[(a_ix+\lambda)^{m+\frac{q-1}{d}}](a_1) \cdot E^{(0)}[(a_i^{-1}x-1)^m](a_1)\bigg)\\
&=\sum_{i=1}^n c_i (a_1a_i+\lambda)^{m+\frac{q-1}{d}} a_i^{-m}+ \binom{m+\frac{q-1}{d}}{m}  \bigg(\sum_{i=1}^n c_i a_i^m (a_1a_i+\lambda)^{\frac{q-1}{d}} (a_i^{-1}a_1-1)^{m}\bigg).
\end{align*}
Since $a_1a_i+\lambda \neq 0$ for each $1 \leq i \leq n$, we have
$(a_1a_i+\lambda)^{\frac{q-1}{d}}=1$ for $i>1$, and thus
$$
(a_1a_i+\lambda)^{\frac{q-1}{d}} (a_i^{-1}a_1-1)=a_i^{-1}a_1-1
$$
for all $i$. Since $a_1^2+\lambda \notin S_d \cup \{0\}$, we have
$$
(a_1^2+\lambda)^m \bigg((a_1^2+\lambda)^{\frac{q-1}{d}}-1\bigg) \neq 0.
$$
Putting these altogether into the computation of $E^{(m)} f (a_1)$, we have
\begin{align*}
E^{(m)} f (a_1)
&=\sum_{i=1}^n c_i (a_1a_i+\lambda)^{m+\frac{q-1}{d}} a_i^{-m}
+\binom{m+\frac{q-1}{d}}{m}  \sum_{i=1}^n c_i a_i^m (a_i^{-1}a_1-1)^{m}\\
&=c_1 a_1^{-m} \bigg((a_1^2+\lambda)^{m+\frac{q-1}{d}}-(a_1^2+\lambda)^m\bigg)+\sum_{i=1}^n c_i (a_1a_i+\lambda)^{m} a_i^{-m} \\
&+ \binom{m+\frac{q-1}{d}}{m} \sum_{k=0}^m \binom{m}{k} a_1^{m-k}(-1)^k  \bigg(\sum_{i=1}^n c_i a_i^k\bigg)\\
&=c_1 a_1^{-m} (a_1^2+\lambda)^m \bigg((a_1^2+\lambda)^{\frac{q-1}{d}}-1\bigg)
+ \sum_{k=0}^m \binom{m}{k} a_1^{k}\lambda^{m-k} \bigg(\sum_{i=1}^n c_i a_i^{k-m}\bigg)
\\
&=c_1 a_1^{-m} (a_1^2+\lambda)^m \bigg((a_1^2+\lambda)^{\frac{q-1}{d}}-1\bigg)
\neq 0,
\end{align*}
where we used the fact $c_1 \neq 0$. In particular, $f$ is not identically zero. 

In conclusion, $f$ is a non-zero polynomial with degree at most $\frac{q-1}{d}+2m$, and \cref{lem:multiplicity} implies that each of $a_1,a_2, \ldots a_n$ is a root of $f$ with multiplicity at least $m$. Recall that $m=n/2-1$.
It follows that
$$
\frac{n(n-2)}{2} =mn  \leq \operatorname{deg}f \leq \frac{q-1}{d}+2m= \frac{q-1}{d}+n-2,
$$
that is, we have $(n-2)^2 \leq \frac{2(q-1)}{d}.$
Therefore, $n \leq \sqrt{2(q-1)/d}+2$. This finishes the proof. \end{proof}

\subsection{Applications to generalized Diophantine tuples over finite fields}
\label{Subsec: applications}

In this subsection, we illustrate how to apply \cref{thm: stepanovea} and \cref{thm: restricted} for obtaining improved upper bounds on the size of a generalized Diophantine tuple or a strong generalized Diophantine tuple over $\mathbb{F}_{q}$, when $q=p$ is a prime and $q$ is a square.

\begin{proof}[Proof of \textup{\cref{mainthm2}}]
(1) Let $A \subset \F_p^*$ with property $SD_d(\lambda, \F_p)$, that is,  $AA+\lambda \subset S_d \cup \{0\}$.
\cref{thm: stepanovea} implies that 
$$
|A|^2 \leq |S_d|+|A \cap (-\lambda A^{-1})|+|A|-1 \leq |S_d|+2|A|-1.
$$
It follows that $(|A|-1)^2 \leq |S_d|$. If $\lambda \in S_d$, we have a stronger upper bound:
$$|A|^2 \leq |S_d|+|A \cap (-\lambda A^{-1})|-1 \leq |S_d|+|A|-1.$$  
It follows that $(|A|-\frac{1}{2})^2 \leq |S_d|-\frac{3}{4}$.

(2) Let $A \subset \F_p^*$ with property $D_d(\lambda, \F_p)$, that is, $A \hat{\times} A+\lambda \subset S_d \cup \{0\}$. If $AA+\lambda \subset S_d \cup \{0\}$, then (1) implies that $|A|\leq \sqrt{p/d}+1$ and we are done. If $AA+\lambda \not \subset S_d \cup \{0\}$, then \cref{thm: restricted} implies the required upper bound.
\end{proof}

\begin{remark}\label{remark:weakerbound} 
\cref{thm: stepanovea} can be used to deduce a weaker upper bound of the form $2\sqrt{p/d}+O(1)$ for \cref{mainthm2} (2). Let $A \subset \F_p^*$ such that $A\hat{\times} A+\lambda \subset S_d \cup \{0\}$. We can write $A=B \sqcup C$ such that $|B|$ and $|C|$ differ by at most $1$. Note that since $B$ and $C$ are disjoint, we have $BC+\lambda \subset A \hat{\times} A+\lambda \subset S_d \cup \{0\}$ and thus \cref{thm: stepanovea} implies that $|B||C| \leq p/d+|B|+|C|$, which further implies that $|A| \leq 2\sqrt{p/d}+O(1)$. Note that such a weaker upper bound is worse than the trivial upper bound from character sums (\cref{mainprop1}) when $d=2,3$, and this is one of our main motivations for establishing the bound $\sqrt{2p/d}+O(1)$ in \cref{mainthm2} (2).
\end{remark}

Next, we consider the case $q$ is a square. First we establish a non-trivial upper bound on $MSD_d(\lambda, \F_q)$ and $MD_d(\lambda,\F_q)$ under some minor assumption. While these new bounds only improve the trivial upper bound from character sums (\cref{mainprop1}) slightly, we will see these new bounds are sometimes sharp in the proof of \cref{mainthm3}. To achieve our goal, we need the following special case of Kummer's theorem \cite{K52}.

\begin{lemma}\label{thm: Kummer}
Let $p$ be a prime and $m,n$ be positive integers. 
If there is no carry between the addition of $m$ and $n$ in base-$p$, then $\binom{m+n}{n}$ is not divisible by $p$.
\end{lemma}

\begin{theorem}\label{thm: square}
Let $q$ be a prime power and a square, and let $\lambda \in \F_q^*$. 
\begin{enumerate}
\item[\textup{(1)}] Let $d \geq 2$ be a divisor of $(q-1)$. Let $r$ be the remainder of $\frac{q-1}{d}$ divided by $p\sqrt{q}$. If $r\leq (p-1)\sqrt{q}$, then $MSD_d(\lambda, \F_q)\leq \sqrt{q}-1$.
\item[\textup{(2)}] Let $q \geq 25$ and let $d \geq 3$ be a divisor of $(q-1)$.    
Let $r$ be the remainder of $\frac{q-1}{d}$ divided by $p\sqrt{q}$. If $r\leq (p-1)\sqrt{q}$, then $MD_d(\lambda, \F_q)\leq \sqrt{q}-1$.
\end{enumerate}
\end{theorem}

\begin{proof}
(1) Since $r\leq (p-1)\sqrt{q}$, there is no carry between the addition of $r-1$ and $\sqrt{q}$ in base-$p$. 
Thus, there is no carry between the addition of $\frac{q-1}{d}-1$ and $\sqrt{q}$ in base-$p$. 
It follows from \cref{thm: Kummer} that
$$
\binom{\sqrt{q}-1+\frac{q-1}{d}}{\sqrt{q}}\not \equiv 0 \pmod p.
$$

Let $A \subset \F_q^*$ with property $SD_d(\lambda, \F_q)$ such that $|A|=MSD_d(\lambda, \F_q)$. Note that \cref{mainprop1} implies that $|A|\leq \sqrt{q}$. For the sake of contradiction, assume that $|A|=\sqrt{q}$. Note that $AA+\lambda \subset S_d \cup \{0\}$ and  
$$
\binom{|A|-1+\frac{q-1}{d}}{|A|}=\binom{\sqrt{q}-1+\frac{q-1}{d}}{\sqrt{q}}\not \equiv 0 \pmod p,
$$
it follows from \cref{thm: stepanovea} that 
$$
|A|^2 \leq |S_d|+|A \cap (-\lambda A^{-1})|+|A|-1 \leq |S_d|+2|A|-1,
$$
that is, $|A|\leq \sqrt{|S_d|}+1<\sqrt{q}$, a contradiction. This completes the proof.

(2) Let $A \subset \F_q^*$ with property $D_d(\lambda, \F_q)$ such that $|A|=MD_d(\lambda, \F_q)$. Then $A\hat{\times} A+\lambda \subset S_d \cup \{0\}$. If $AA+\lambda \subset S_d \cup \{0\}$, we just apply (1). Next assume that $AA+\lambda \not \subset S_d \cup \{0\}$, then \cref{thm: restricted} implies that
$$
|A| \leq \sqrt{\frac{2(q-1)}{d}}+4 \leq \sqrt{\frac{2(q-1)}{3}}+4\leq \sqrt{q}
-1,
$$
provided that $q \geq 738$. When $25 \leq q \leq 737$, we have used SageMath to verify the theorem. 
\end{proof}

Now we are ready to prove \cref{mainthm3}, which determines the maximum size of an infinitely family of generalized Diophantine tuples and strong generalized Diophantine tuples over finite fields.

\begin{proof}[Proof of \textup{\cref{mainthm3}}] 
In both cases, the upper bound $\sqrt{q}-1$ follows from \cref{thm: square}. To show that $\sqrt{q}-1$ is a lower bound, we observe that $A=\alpha \F_{\sqrt{q}}^*$ has property $SD_{d}(\lambda,\mathbb{F}_q)$  (and therefore $D_{d}(\lambda,\mathbb{F}_q)$). Indeed, $AA+\lambda=\alpha^2 \F_{\sqrt{q}}^*+\lambda \subset \alpha^2 \F_{\sqrt{q}} \subset S_d \cup \{0\}$ since $\alpha^2 \in S_d$ and $\F_{\sqrt{q}}^* \subset S_d$ (from the assumption $d \mid (\sqrt{q}+1)$).
\end{proof}

\begin{remark}
Our SageMath code indicates that the last statement of \cref{mainthm3} does not hold when $d=2$ and $q=9,25,49$, when $d=3$ and $q=4,16$, and when $d=4$ and $q=9$. We  conjecture the same statement holds for $d=2$, provided that $q$ is sufficiently large.
\end{remark}

So far we have only considered special cases of applying \cref{thm: stepanovea}. In general, to apply \cref{thm: stepanovea}, the assumption on the binomial coefficient in the statement of \cref{thm: stepanovea} might be tricky to analyze. However, if the base-$p$ representation of $\frac{q-1}{d}$ behaves ``nicely" (for example, if the order of $p$ modulo $d$ is small, then the base-$p$ representation is periodic with a small period), then it is still convenient to apply \cref{thm: stepanovea}. As a further illustration, we prove the following theorem. Note that the new bound is of the same shape as that in \cref{mainthm2} (2), so it can be viewed as a generalization of \cref{mainthm2} (2) as changing a prime $p$ to an arbitrary power of $p$, provided that $d \mid (p-1)$.  
\begin{theorem}\label{mainthm2.5}
Let $d \geq 2$, and let $q$ be a power of $p$ such that $d \mid (p-1)$. Then $MD_d(\lambda, \F_q) \leq \sqrt{2(q-1)/d}+4$ for any $\lambda \in \F_q^*$.   
\end{theorem}
\begin{proof}
Let $B \subset \F_q^*$ with property $D_d(\lambda, \F_q)$, that is,  $B\hat{ \times} B+\lambda \subset S_d \cup \{0\}$. If $BB+\lambda  \nsubseteq S_d \cup \{0\}$, we are done by \cref{thm: restricted}.
Thus, we may assume that $BB+\lambda \subset S_d \cup \{0\}$. It suffices to show $|B| \leq \sqrt{2(q-1)/d}+4$. To achieve that, we try to find an arbitrary subset $A$ of $B$ such that $\binom{|A|-1+\frac{q-1}{d}}{|A|} \not \equiv 0 \pmod p$ and $|A|$ is as large as possible. With such a subset $A$, we have $AB+\lambda \subset S_d \cup \{0\}$ so that we can apply \cref{thm: stepanovea}. In the rest of the proof, we aim to find such an $A$ with $|A|\geq |B|/2$ so that, from \cref{thm: stepanovea}, we can deduce
$$
\frac{|B|^2}{2} \leq \frac{q-1}{d}+2|B|-1 \implies |B| \leq \sqrt{\frac{2(q-1)}{d}+2}+2<\sqrt{\frac{2(q-1)}{d}}+4.
$$

Write $|B|-1=(c_k, c_{k-1}, \ldots, c_1,c_0)_p$ in base-$p$, that is, $|B|-1=\sum_{i=0}^k c_ip^i$ with $0 \leq c_i \leq p-1$ for each $0 \leq i \leq k$ and $c_k \geq 1$. Next, we construct $A$ according to the size of $c_k$.

\textit{Case 1.} $c_k \leq p-1-\frac{p-1}{d}$. In this case, let $A$ be an arbitrary subset of $B$ with $|A|-1=(c_k,0, \ldots, 0)_p$, that is, $|A|=c_kp^k+1$. It is easy to verify that $\binom{|A|-1+\frac{q-1}{d}}{|A|} \not \equiv 0 \pmod p$ using \cref{thm: Kummer}. Since $|B| \leq (c_k+1)p^k$, it also follows readily that $|A|\geq |B|/2$.

\textit{Case 2.} $c_k > p-1-\frac{p-1}{d}$. In this case, let $A$ be an arbitrary subset of $B$ with $$|A|-1=\bigg(\frac{(d-1)(p-1)}{d},\frac{(d-1)(p-1)}{d}, \ldots, \frac{(d-1)(p-1)}{d}\bigg)_p,$$ that is, $|A|=\frac{(d-1)(p-1)}{d} \cdot \sum_{i=0}^k p^i +1$. Again, it is easy to verify that $\binom{|A|-1+\frac{q-1}{d}}{|A|} \not \equiv 0 \pmod p$ using \cref{thm: Kummer}. Since $d \geq 2$, it follows that $2|A| \geq (p-1)\sum_{i=0}^k p^i +2= p^{k+1}+1>|B|$.
\end{proof}

\begin{remark}\label{mainthm2.5_remark}
Under the same assumption, the proof of \cref{mainthm2.5} can be refined to obtain improved upper bounds on $MSD_d(\lambda, \F_q)$. In particular, if $d,r \geq 2$ are fixed, and $p \equiv 1 \pmod d$ is a prime, then as $p \to \infty$, we can show that $MSD_d(\lambda, \F_{p^{2r-1}}) \leq (1+o(1))\sqrt{p^{2r-1}/d}$ uniformly among $\lambda \in \F_{p^{2r-1}}^{*}$. 
Indeed, if $B \subset \F_q^*$ with property $D_d(\lambda, \F_q)$ with $q=p^{2r-1}$ and $\lambda \in \F_{q}^{*}$, then we can assume without loss of generality that $\sqrt{q/d}<|B|$. Otherwise, we are done. Note that $|B|<\sqrt{q}+O(1)$ by \cref{mainprop1}. Following the notations used in the proof of \cref{mainthm2.5}, we have $\sqrt{p/d}-1\leq c_k \leq \sqrt{p}$ and thus we are always in Case 1, and the same construction of $A$ gives $|A|=(1-o(1))|B|$ as $p \to \infty$.
Thus, \cref{thm: stepanovea} gives $|B| \leq (1+o(1))\sqrt{q/d}$. 
\end{remark}

\section{Improved upper bounds on the largest size of generalized Diophantine tuples over integers}
\label{sec: GDM}

\subsection{Proof of \cref{mainthm1}}\label{subsec: 5.1}
In this subsection, we improve the upper bounds on the largest size of generalized Diophantine tuples with property $D_{k}(n)$. We first recall that for each $n \ge 1$ and $k \ge 2$,
\[M_{k}(n)=\sup \{|A| \colon A \text{ satisfies property }D_{k}(n)\}.\]
For $k \geq 2$, \textcolor{black}{we defined the constant in the introduction} 
\begin{equation}\label{alpha_k} 
\eta_k=\min_{\mathcal{I}} \frac{|\mathcal{I}|}{T_\mathcal{I}^2},
\end{equation}
where the minimum is taken over all nonempty subset $\mathcal{I}$ of 
$$\{1 \leq i \leq k: \gcd(i,k)=1, \gcd(i-1,k)>1\},$$ and
\begin{equation}
\label{D}
    T_{\mathcal{I}}=\sum_{i \in \mathcal{I}} \sqrt{\gcd(i-1,k)}.
\end{equation}    

Here is the proof of our main theorem, \cref{mainthm1}.

\begin{proof}[Proof of \textup{\cref{mainthm1}}]
Let $A=\{a_{1},a_{2},\ldots,a_{m}\}$ be a generalized Diophantine $m$-tuple with property $D_{k}(n)$ and $k\geq 3$. \textcolor{black}{Given the assumption that $\log k=O(\sqrt{\log \log n})$, \cref{prop:effective1} implies that the contribution of $a_i$ with $a_i>n^{\frac{k}{k-2}}$ is $|A\cap (n^{\frac{k}{k-2}},\infty)|=O(\log k \log \log k)$ is negligible. Thus, we can assume that $A \subset [1, n^{\frac{k}{k-2}}]$.} Let $\mathcal{I}$ be a nonempty subset of $\{1 \leq i \leq k: \gcd(i,k)=1, \gcd(i-1,k)>1\}$, such that the ratio $|\mathcal{I}|/T_\mathcal{I}^2$ in equation \cref{alpha_k} is minimized by $\mathcal{I}$. 
In other words, we have $\eta_k=|\mathcal{I}|/T_{\mathcal{I}}^2$, where 
$$
T=T_{\mathcal{I}}=\sum_{i \in \mathcal{I}} \sqrt{\gcd(i-1,k)}.
$$
To apply the Gallagher sieve inequality (\cref{larger}), we set \textcolor{black}{$N=n^{\frac{k}{k-2}}$} and define the set of primes 
$$\mathcal{P}=\{p: p \equiv i \pmod k \text{ for some }i \in \mathcal{I}\} \setminus \{p: p \mid n\}.$$
For each prime $p \in \mathcal{P}$, denote by $A_{p}$ the image of $A \pmod{p}$ and let $A_p^*=A_p \setminus \{0\}$. 

Let $p \in \mathcal{P}$. We can naturally view $A_p^*$ as a subset of $\F_p^*$. Since $A$ has property $D_k(n)$, it follows that $A_p^* \hat{ \times} A_p^*+ n \subset \{x^k: x \in \F_p^*\} \cup \{0\}$. Note that $\{x^k: x \in \F_p^*\}$ is the multiplicative subgroup of $\F_p^*$ with order $\frac{p-1}{\gcd(p-1,k)}$. Since $\gcd(p-1,k)>1$ and $p \nmid n$,  \cref{mainthm2} (2) implies that 
\[ |A_{p}| \leq |A_p^*|+1 \leq \sqrt{\frac{2(p-1)}{\gcd(p-1,k)}}+5.\]
Set $Q=2(\frac{\phi(k)\log N}{T})^2$. Applying Gallagher's larger sieve, we obtain that 
\begin{equation}\label{eq:bound}
|A| \le \frac{\sum_{p \in \mathcal{P},p \le Q}\log p - \log N}{\sum_{p \in \mathcal{P},p \le Q} \frac{\log p}{|A_{p}|}-\log N}.
\end{equation}
Let $c$ be the constant from \cref{partial-summation}. For the numerator on the right-hand side of inequality \cref{eq:bound}, we have
\begin{align*}
\sum_{p \in \mathcal{P},p \le Q}\log p - \log N
& \leq \sum_{i \in \mathcal{I}}\bigg(\sum_{\substack{p\equiv i \text{ mod }k, \\p \le Q}}\log p\bigg) - \log N \\
&=\frac{|\mathcal{I}|Q}{\phi(k)}+O\bigg(|\mathcal{I}|Q\exp(-c\sqrt{\log Q})\bigg)-\log N.
\end{align*}
Next, we estimate the denominator on the right-hand side of inequality \cref{eq:bound}. 
Note that $|\mathcal{I}|\le T=\sum_{i \in \mathcal{I}} \sqrt{\gcd(i-1,k)}$. Then we have $T \le |\mathcal{I}| \sqrt{k} \le \phi(k)\sqrt{k}$, and so $\phi(k)/T \ge 1/\sqrt{k}$.
Since $k=(\log N)^{o(1)}$, we deduce $Q > 2(\log N)^{2-o(1)}$. 
Thus we have $k=Q^{o(1)}$.
This, together with \cref{partial-summation} and \cref{p mid n}, deduces that for each $i \in \mathcal{I}$,
\begin{align*}
\sum_{\substack{p \in \mathcal{P}, p \le Q\\  p \equiv i \text{ mod } k}} \frac{\log p}{|A_{p}|}
& \geq \sum_{\substack{p \in \mathcal{P}, p \le Q \\  p \equiv i \text{ mod } k}} \frac{\log p}{\sqrt{\frac{2(p-1)}{\gcd(i-1,k)}}+5}\\
&=\sum_{\substack{p \le Q\\  p \equiv i \text{ mod } k}} \frac{\log p}{\sqrt{\frac{2p}{\gcd(i-1,k)}}}+O\bigg(\sum_{p \le Q} \frac{k\log p}{p}\bigg)+O\bigg(\sum_{p \mid n} \frac{\sqrt{k}\log p}{\sqrt{p}}\bigg)\\
&=\sqrt{\frac{\gcd(i-1,k)}{2}} \sum_{\substack{p \le Q\\  p \equiv i \text{ mod } k}} \frac{\log p}{\sqrt{p}}+O(k\log Q)+O(k (\log n)^{1/2})\\
&=\frac{\sqrt{2Q\gcd(i-1,k)}}{\phi(k)} +O\left( \sqrt{Q} \sqrt{\gcd(i-1,k)}
 \exp(-c\sqrt{\log Q}) \right).
\end{align*}
Thus we have
\begin{align*}
|A| 
&\le \frac{\sum_{p \in \mathcal{P},p \le Q}\log p - \log N}{\sum_{p \in \mathcal{P},p \le Q} \frac{\log p}{|A_{p}|}-\log N}\\
&\le \frac{\frac{|\mathcal{I}|Q}{\phi(k)}+O(|\mathcal{I}|Q\exp(-c\sqrt{\log Q}))-\log N}{\sum_{i \in \mathcal{I}}\bigg( \sum_{\substack{p \in \mathcal{P}, p \le Q\\  p \equiv i \text{ mod } k}} \frac{\log p}{|A_{p}|}\bigg) -\log N}\\
&\le \frac{\frac{|\mathcal{I}|Q}{\phi(k)}+O(|\mathcal{I}|Q\exp(-c\sqrt{\log Q}))-\log N}{\frac{T\sqrt{2Q}}{\phi(k)}+O\left(T
 \sqrt{Q} \exp(-c\sqrt{\log Q}) \right)-\log N}.
\end{align*}

Finally, recall that $Q=2(\frac{\phi(k)\log N}{T})^2$. It follows that
\begin{align*}
|A|&\leq \frac{2|\mathcal{I}|\phi(k)(\frac{\log N}{T})^2+O\left(|\mathcal{I}| (\frac{\phi(k)\log N}{T})^2\exp \Big(-c\sqrt{\log (\frac{\phi(k)\log N}{T})}\Big)\right)}{2\log N+O\left(\phi(k)\log N\exp \Big(-c\sqrt{\log (\frac{\phi(k)\log N}{T})}\Big)\right)-\log N}\\
&= \frac{\frac{2|\mathcal{I}|\phi(k)}{T^2} \log N+O\left(\frac{|\mathcal{I}|\phi(k)^2 \log N}{T^2} \exp \Big(-c\sqrt{\log (\frac{\phi(k)\log N}{T})}\Big)\right)}{1+O\left(\phi(k)
\exp \Big(-c\sqrt{\log (\frac{\phi(k)\log N}{T})}\Big)\right)}.
\end{align*}
Recall that \textcolor{black}{$N=n^{\frac{k}{k-2}}$}. Thus, to obtain our desired result, we need to show
\begin{align*}
|A|\leq  \frac{(1+o(1))\frac{2|\mathcal{I}|\phi(k)}{T^2} \log N}{1-o(1)},
\end{align*}
and it suffices to show that
\[\phi(k)\exp \Big(-c\sqrt{\log (\frac{\phi(k)\log N}{T})}\Big)=o(1),\]
as $N \to \infty$, or equivalently,
\[\log k- c\sqrt{\log \frac{\phi(k)}{T}+\log \log N} \to -\infty,\]
as $N \to \infty$. We notice that  $\phi(k)/T \ge 1/\sqrt{k}$.
Let $c'=c/2$. Then the assumption $\log k\leq c'\sqrt{\log \log n}<c'\sqrt{\log \log N}$ implies 
\[\log \log N + \log \frac{\phi(k)}{T} \ge \log \log N - \frac{1}{2}\log k = (1-o(1)) \log \log N,\]
and
\[\log k  - c\sqrt{\log \frac{\phi(k)}{T}+\log \log N} \leq -(c'-o(1))\log \log N \to -\infty,\]
as required.
\end{proof}

\begin{remark}\label{remark:6}
Note that when $\mathcal{I}=\{1\}$, that is to say, when we only consider primes $p$ such that $p \equiv 1 \pmod k$ for applying the Gallagher inequality, the condition $p \equiv 1 \pmod k$ guarantees that the $k$-th powers are indeed $k$-th powers modulo $p$. We have $T=T_\mathcal{I}=\sqrt{k}$, thus we trivially have $\eta_k \leq \frac{1}{k}$ in view of equation \cref{alpha_k}. In particular, if $k$ is fixed and $n \to \infty$, \cref{mainthm1} implies that
\textcolor{black}{
\begin{equation}\label{naive}
M_k(n) \leq \frac{(2+o(1))\phi(k)}{k-2} \log n, 
\end{equation}
}
which already provides a substantial improvement on the best-known upper bound $M_k(n)\leq (3\phi(k)+o(1))\log n$ whenever $k \geq 3$ given in \cite{DKM22}.
Moreover, note that $\frac{\phi(k)}{k}$ can be as small as $O(\frac{1}{\log \log k})$ when $k$ is the product of distinct primes \cite[Theorem 2.9]{MV07}. 
Thus, in view of \cref{mainthm1}, the inequality \cref{naive} already shows there is $k=k(n)$ such that $\log k \asymp \sqrt{\log \log n}$ and
\begin{equation}\label{logloglogn}
M_k(n)\ll \frac{\log n}{\log \log k} \ll \frac{\log n}{\log \log \log n}.   
\end{equation}
Note that \cref{logloglogn} already breaks the $\log n$ barrier. On the other hand, we can still use other primes $p$ such that $\gcd(p-1,k)>1$ for which $k$-th powers are in fact $\gcd(k,p-1)$-th powers modulo $p$ when we apply the Gallagher sieve inequality. We can take advantage of the improvement on the upper bound of $M_{k}(n)$. In the next two subsections, we further provide a significant improvement on inequality \cref{logloglogn}.
\end{remark}

Next, we define a \emph{strong Diophantine $m$-tuple with property $SD_{k}(n)$} to be a set $\{a_1,\ldots, a_m\}$ of $m$ distinct positive integers such that $a_{i}a_{j}+n$ is a $k$-th power for any choice of $i$ and $j$.
We have a stronger upper bound for the size of a strong Diophantine tuple with property $SD_k(n)$. We define 
\[MS_{k}(n)=\sup \{|A| \colon A\subset{\mathbb{N}} \text{ satisfies the property }SD_{k}(n)\}.\]
\begin{theorem}\label{mainthm1_strong}
There is a constant $c'>0$, such that as $n \to \infty$,
$$
MS_k(n)\leq \bigg(\textcolor{black}{\frac{k}{k-2}}+o(1)\bigg) \ \eta_k  \phi(k) \log n,
$$
holds uniformly for positive integers $k,n \geq 3$ such that $\log k \leq c'\sqrt{\log \log n}$. Moreover, if $k$ is even, under the same assumption \textcolor{black}{(including the case $k=2$)}, we have the stronger bound
$$
MS_k(n)\leq \min \{\big(1+o(1)\big)  \eta_k  \phi(k) \log n, \tau(n)\},
$$    
where $\tau(n)$ is the number of divisors of $n$. 
\end{theorem}
\begin{proof}
The proof is very similar to the proof of \cref{mainthm1} and we follow all the notations and steps as in the proof of \cref{mainthm1}, apart from the minor modifications stated below.

We prove the first part. For each $p \in \mathcal{P}$, we have the stronger upper bound that $|A_p| \leq \sqrt{\frac{(p-1)}{\gcd(p-1,k)}}+2$ by \cref{mainthm2} (2). To optimize the upper bound obtained from Gallagher's larger sieve, we instead set $Q=(\frac{\phi(k)\log N}{T})^2$. 

Next, we assume that $k$ is even and prove the second part. Notice that for each $x \in A$, there is a positive integer $y$, such that $x^2+n=y^2$. Thus, $|A|$ is bounded by the number of positive integral solutions to the equation $x^2+n=y^2$, which is at most $\tau(n)$. On the other hand, this also implies that all the elements in $A$ are at most $n$. Thus, we can set $N=n$ instead and obtain the stronger upper bound.
\end{proof}

\subsection{Proof of \cref{mainthm1.5}}\label{subsec: size of Mk}

In this subsection, by finding a more refined upper bound on $\eta_k$ in equation \cref{alpha_k}, we show that the same approach significantly improves the upper bound of $M_{k}(n)$ in inequality \cref{logloglogn} when $k$ is the product of the first few distinct primes. 

We label all the primes in increasing order so that $2=p_1<p_2<\cdots<p_{\ell}<\cdots$. Let $P_{\ell}=\prod_{i=1}^{\ell} p_i$ be the product of first $\ell$ primes. Let $\mathcal{I}_1=\{1\}$. For $\ell \geq 1$, we define $\mathcal{I}_{\ell+1}$ inductively:
\begin{equation}
\label{C_ell}
\mathcal{I}_{\ell+1}=\{i+j P_{\ell}: i \in \mathcal{I}_{\ell}, 0 \leq j<p_{\ell+1}, p_{\ell+1} \nmid (i+j P_{\ell})\}.
\end{equation}
We note that $\mathcal{I}_{\ell}\subset \mathcal{I}_{\ell+1}$ for any $\ell \geq 1$. Also, it is clear that 
\begin{align}
\label{eq: Cl}
|\mathcal{I}_{\ell+1}|=|\mathcal{I}_{\ell}|(p_{\ell+1}-1).
\end{align}
\begin{lemma}
Following the above definitions, we have
    \begin{equation}
\label{claim}
    \mathcal{I}_\ell \subset \{1 \leq x \leq P_{\ell}: \gcd(x, P_{\ell})=1, \gcd(x-1, P_{\ell})>1\}.
\end{equation}    
\end{lemma}

\begin{proof}
We give an inductive proof. When $\ell=1$, the inclusion \cref{claim} holds. We assume that \cref{claim} holds for some $\ell \geq 1$. Let $x=i+jP_{\ell} \in \mathcal{I}_{\ell+1}$.  
By the assumption, we have $\gcd(i, P_{\ell})=1$, and it follows that $\gcd(x, P_{\ell+1})=\gcd(x,P_{\ell})\gcd(x,p_{\ell +1})=\gcd(i, P_{\ell})\gcd(x, p_{\ell+1})=1.$
This proves the claim.
\end{proof}
Furthermore, we introduce the following notation which is similar to the previously introduced on equation \cref{D}. For each $\ell\geq 1$, we let
$$
T_{\ell}=\sum_{y \in \mathcal{I}_\ell} \sqrt{\gcd(y-1, P_{\ell})}.
$$
Note that $T_1=\sqrt{2}$. We also establish a recurrence relation on the sequence.
\begin{lemma}\label{lem: Dl}
The sequence $(T_{\ell})_{\ell\geq 1}$ satisfies the recurrence relation
\begin{align}\label{eq: Dl}
T_{\ell+1}=T_{\ell}(p_{\ell+1}-2+\sqrt{p_{\ell+1}}).
\end{align}
\end{lemma}
\begin{proof}
We have
    \begin{align*}
T_{\ell+1}
&=\sum_{i \in \mathcal{I}_{\ell}} \sum_{\substack{0 \leq j<p_{\ell+1}\\ p_{\ell+1} \nmid (i+j P_{\ell})}}\sqrt{\gcd(i+jP_{\ell}-1, P_{\ell+1})}\\
&=\sum_{i \in \mathcal{I}_{\ell}} \sum_{\substack{0 \leq j<p_{\ell+1}\\ p_{\ell+1} \nmid (i+j P_{\ell})}}\sqrt{\gcd(i+jP_{\ell}-1, P_{\ell})} \sqrt{\gcd(i+jP_{\ell}-1, p_{\ell+1})}\\
&=\sum_{i \in \mathcal{I}_{\ell}} \sqrt{\gcd(i-1, P_{\ell})} \bigg(\sum_{\substack{0 \leq j<p_{\ell+1}\\ p_{\ell+1} \nmid (i+j P_{\ell})}} \sqrt{\gcd(i+jP_{\ell}-1, p_{\ell+1})}\bigg).
\end{align*}
It is easy to show that the inner sum consists of $(p_{\ell+1}-2)$ many $1$ and a single $\sqrt{p_{\ell+1}}$. It follows that
$$
T_{\ell+1}=(p_{\ell+1}-2+\sqrt{p_{\ell+1}})\sum_{i \in \mathcal{I}_{\ell}} \sqrt{\gcd(i-1, P_{\ell})}=T_{\ell} (p_{\ell+1}-2+\sqrt{p_{\ell+1}}),
$$
proving the lemma.
\end{proof}

We are now ready to prove \cref{mainthm1.5}.

\begin{proof}[Proof of \textup{\cref{mainthm1.5}}]
For each $n$, we choose $k=k(n)=P_\ell$, where $\ell=\ell(n)$ is the largest integer such that $\log P_\ell<c' \sqrt{\log \log n}$. It follows that $\log k=\log P_\ell \asymp \sqrt{\log \log n}$. 
Thus, using equations \cref{eq: Cl} and $\cref{eq: Dl}$, we have
$$
\eta_k \phi(k) 
\leq \frac{|\mathcal{I}_\ell| \phi(P_\ell)}{T_{\ell}^2}
=\prod_{p \leq p_{\ell}} \frac{(p-1)^2}{(p-2+\sqrt{p})^2}.
$$
Note that for each prime $p$, it is easy to verify that
$
\frac{p-1}{p-2+\sqrt{p}} \leq 1-\frac{1}{\sqrt{p}}.
$
Recall that the inequality $e^x \geq 1+x$ holds for all real $x$, and a standard application of partial summation gives
$$
\sum\limits_{p \leq x} \frac{1}{\sqrt{p}} = \frac{2\sqrt{x}}{\log x} + O\left(\frac{\sqrt{x}}{\log ^2x}\right).
$$
Also, the prime number theorem implies that 
$$
\log P_\ell=\sum_{p \leq p_{\ell}} \log p=\theta(p_{\ell})=(1+o(1))p_{\ell}
$$
and thus $p_{\ell}=(1+o(1))\log P_{\ell}$. Putting the above estimates altogether, we have
\begin{align*}
\eta_k \phi(k) 
&\leq \prod_{p  \leq p_{\ell}} \bigg(1-\frac{1}{\sqrt{p}}\bigg)^2
\leq \exp \bigg(-2\sum_{p \leq p_{\ell}} \frac{1}{\sqrt{p}}\bigg)\\ &=\exp\bigg(-\frac{(4+o(1))\sqrt{p_{\ell}}}{\log p_{\ell}}\bigg)
=\exp\bigg(-\frac{(4+o(1))\sqrt{\log P_\ell}}{\log \log P_\ell}\bigg)\\
&\leq \exp\bigg(-\frac{c''(\log \log n)^{1/4}}{\log \log \log n}\bigg).
\end{align*}
for some absolute constant $c''>0$.
It follows from \cref{mainthm1} that 
\[
M_k(n) \ll \eta_k\phi(k) \log n \ll \exp\bigg(-\frac{c''(\log \log n)^{1/4}}{\log \log \log n}\bigg) \log n. 
\]
\end{proof}

\subsection{An upper bound on \texorpdfstring{$\eta_{k}$}{}} \label{sec: approximation}
In this subsection, we deduce a simple upper bound of $\eta_k$. It turns out that this upper bound well approximates $\eta_k$ empirically.

\begin{theorem}\label{refined_alpha_k}
For any $k \ge 2$, we have $$\eta_{k} \le \mu_{k},$$ where $\mu_{k} = R_{k} \cdot \min \{\beta(p^\alpha): p^\alpha \vert \vert k\}$ with
\[R_{k}=\prod_{p^{\alpha} \mid \mid k} \frac{(p-1) p^{\alpha-1}}{\big(p^{\alpha} - p^{\alpha -1} - p^{(\alpha-1)/2}  + p^{\alpha -1/2}\big)^2},\]
and
\[\beta(p^\alpha) = \frac{\big(p^{\alpha} - p^{\alpha -1} - p^{(\alpha-1)/2}  + p^{\alpha -1/2}\big)^2}{(p-1)\big(- p^{(\alpha-1)/2} + p^{\alpha-1}  + p^{\alpha-1/2}\big)^2}.\]
\end{theorem}
\begin{proof}
We denote $k=\prod_{j=1}^{\ell} p_j^{\alpha_j}$, where $p_1,p_2, \ldots, p_\ell$ are distinct primes factors of $k$ such that $$\beta(p_\ell^{\alpha_\ell})=\min\{\beta(p^\alpha): p^\alpha \mid \mid k\}.$$
Define $$\mathcal{I}=\{1 \leq i \leq k: \gcd(k,i)=1, i \equiv 1 \pmod {p_{\ell}}\}, \quad
T_{\mathcal{I}}=\sum_{i \in \mathcal{I}} \sqrt{\gcd(i-1,k)}.
$$
Then $\mathcal{I}$ is obviously a subset of the set 
$\{1 \leq i \leq k: \gcd(i,k)=1, \gcd(i-1,k)>1\}$
consisting of residue classes that can be used in Gallagher's larger sieve in the proof of \cref{mainthm1}.
(In particular, when $p_\ell=2$, $\mathcal{I}$ consists of all the available residue classes with $|\mathcal{I}|=\phi(k)$.) In view of the definition of $\eta_k$, it suffices to show that 
$$\frac{|\mathcal{I}|}{T_{\mathcal{I}}^2}=\mu_k=R_k \cdot \beta(p_\ell^{\alpha_\ell}).$$
We first compute the size of $\mathcal{I}$. Equivalently, we can write 
$$
\mathcal{I}=\{1 \leq i \leq k: i \not \equiv 0 \pmod{p_j} \text{ for each }1 \leq j<\ell, \text{ and }  i \equiv 1 \pmod {p_{\ell}} \},
$$
and hence, we deduce $|\mathcal{I}|=\prod_{j=1}^{\ell-1} (p_j-1)p_j^{\alpha_j-1} \cdot p_{\ell}^{\alpha_\ell-1}$.
In order to compute $T_{\mathcal{I}}$, we first count the number of solutions to $v_{p_j}(i-1)=s$ over $1 \leq i \leq p_j^{\alpha_j}$ such that $p_j \nmid i$ for $0 \leq s \leq \alpha_j$ separately, and then use the Chinese remainder theorem. 
Set
\begin{align*}
C_{j,s}&=\{1 \leq i \leq p_j^{\alpha_j}: i\not\equiv 0 \pmod{p_{j}},~ v_{p_j}(i-1)=s \}, \quad \text{for $0\leq s\leq \alpha_{j}$, $j<\ell$;}\\
C_{\ell,s}&=\{1 \leq i \leq p_{\ell}^{\alpha_{\ell}}: i\equiv 1 \pmod{p_{\ell}},~ v_{p_\ell}(i-1)=s \}, \quad \text{for $0\leq s\leq \alpha_{\ell}$}.
\end{align*}
Note that
\begin{align*}
|C_{j,s}|&=\phi(p_j^{\alpha_j-s}), \quad \quad \quad \quad \text{for $0<s\leq \alpha_{j}$, $j<\ell$;}\\
|C_{j,0}|&=\phi(p_j^{\alpha_j})-p_{j}^{\alpha_{j}-1}, \quad \; \; \text{for $j<\ell$;}\\
|C_{\ell,s}|&=\phi(p_{\ell}^{\alpha_{\ell}-s}), \quad \quad \quad \quad \text{for $0<s\leq \alpha_{\ell}$},
\end{align*}
and $|C_{\ell,0}|=0$. It follows that
\begin{align*}
T_{\mathcal{I}}=\sum_{i \in \mathcal{I}} \sqrt{\gcd(i-1,k)}=\sum_{d \mid k} \sqrt{d} \sum_{\substack{i \in \mathcal{I},\\ \gcd(i-1,k)=d}} 1=\prod_{j=1}^{\ell}\left(\sum_{s=0}^{\alpha_{j}}\sqrt{p_{j}^{s}}|C_{j,s}|\right).
\end{align*}
For each $1 \leq j \leq \ell-1$, we calculate
\begin{align*}
\sum_{s=0}^{\alpha_j} \sqrt{p_j^s} |C_{j,s}|
=\phi(p_{j}^{\alpha_{j}}) - p_{j}^{\alpha_{j}-1} + \sum_{s=1}^{\alpha_{j}}\sqrt{p_{j}^s} \phi(p_{j}^{\alpha_{j}-s})
=p_{j}^{\alpha_{j}} - p_{j}^{\alpha_{j} -1} - p_{j}^{(\alpha_{j}-1)/2}  + p_{j}^{\alpha_{j} -1/2}. 
\end{align*}

Similarly, we have
\begin{align*}
\sum_{s=0}^{\alpha_\ell} \sqrt{p_\ell^s} |C_{\ell,s}| 
& = \sum_{s=1}^{\alpha_\ell} \sqrt{p_\ell^s} \phi(p_\ell^{\alpha_\ell-s})=  - p_{\ell}^{(\alpha_{\ell}-1)/2} + p_{\ell}^{\alpha_{\ell}-1}  + p_{\ell}^{\alpha_{\ell}-1/2}.\\
\end{align*}
Putting these all together, we compute
\begin{align*}
T_{\mathcal{I}}
&=\prod_{j=1}^{\ell-1} \bigg[p_{j}^{\alpha_{j}} - p_{j}^{\alpha_{j} -1} - p_{j}^{(\alpha_{j}-1)/2}  + p_{j}^{\alpha_{j} -1/2} \bigg]\cdot \bigg( - p_{\ell}^{(\alpha_{\ell}-1)/2} + p_{\ell}^{\alpha_{\ell}-1}  + p_{\ell}^{\alpha_{\ell}-1/2} \bigg).
\end{align*}
Hence, 
\begin{align*}
\frac{|\mathcal{I}|}{T_{\mathcal{I}}^2}
&=\prod_{j=1}^{\ell-1} \frac{(p_j-1) p_j^{\alpha_j-1}}{\big(p_{j}^{\alpha_{j}} - p_{j}^{\alpha_{j} -1} - p_{j}^{(\alpha_{j}-1)/2}  + p_{j}^{\alpha_{j} -1/2}\big)^2}\cdot \frac{ p_\ell^{\alpha_\ell-1}}{\big( - p_{\ell}^{(\alpha_{\ell}-1)/2} + p_{\ell}^{\alpha_{\ell}-1}  + p_{\ell}^{\alpha_{\ell}-1/2}\big)^2}.
\end{align*}

\end{proof}

Therefore, \cref{mainthm1} implies 
\begin{corollary}
There is a constant $c'>0$, such that as $n \to \infty$,
$$
M_k(n)\leq \bigg(\textcolor{black}{\frac{2k}{k-2}}+o(1)\bigg) \ \mu_k  \phi(k) \log n,
$$
holds uniformly for positive integers $k,n \geq 3$ such that $\log k \leq c'\sqrt{\log \log n}$. 
\end{corollary}

\begin{remark}
Our computations indicate that when $2 \leq k \leq 100{,}000$, the inequality $\mu_k \leq 2\eta_k$ holds for all but $501$ of them. This numerical evidence suggests that $\mu_k$ provides a good approximation for $\eta_k$ for a generic $k$. Note that the computational complexity for computing $\mu_k$ is the same as that of the prime factorization of $k$: a naive algorithm takes $O(\sqrt{k})$ time. The best theoretical algorithm has running time $O\big(\exp((\log k)^{1/3+o(1)})\big)$ using the general number field sieve \cite{BLP93}. On the other hand, computing $\eta_k$ requires $O(k\log k)$ time; we refer to \cref{appendix} for an algorithm and some computational results. 
\end{remark}

\section{Multiplicative decompositions of shifted multiplicative subgroups}\label{sec: Multiplicative decompositions}

In this section, we present our contributions to \cref{conjecture:MD}. In particular, we make significant progress towards S\'{a}rk\"{o}zy's conjecture (\cref{conjecture:S}). We recall $S_d=S_d(\F_q)=\{x^d: x \in \F_q^*\}$.

\subsection{Applications to S\'{a}rk\"{o}zy's conjecture}
\label{subsec: Sarkozy}

 In this subsection, we show that for almost all primes $p \equiv 1 \pmod d$, the set $(S_d(\F_p)-1) \setminus \{0\}$ cannot be decomposed as the product of two sets non-trivially. This confirms the truth of S\'{a}rk\"{o}zy's conjecture (\cref{conjecture:S})  as well as the truth of its generalization in the generic case (\cref{conjecture:MD}) when the shift of the subgroup is given by $\lambda=1$. 

\begin{theorem}\label{thm:almostall}
Let $d \geq 2$ be fixed. As $x \to \infty$, the number of primes $p \leq x$ such that $p \equiv 1 \pmod d$ and  $(S_d(\F_p)-1) \setminus \{0\}$ can be decomposed as the product of two sets non-trivially (that is, it can be written as the product of two subsets of $\F_p^*$ with size at least 2) is $o(\pi(x))$. 
\end{theorem}
\begin{proof}
Let $\mathcal{P}_d$ be the set of primes $p$ such that $p \equiv 1 \pmod d$ and $(S_d(\F_p)-1) \setminus \{0\}$ admits a non-trivial multiplicative decomposition. By the prime number theorem for arithmetic progressions, it suffices to show that $|\mathcal{P}_d \cap [0,x]|=o(x/\log x)$. 

Let $p \in \mathcal{P}_d$. Then we can write $(S_d-1) \setminus \{0\}$ as the product of two sets $A,B \subset \F_p^*$ such that $|A|, |B| \geq 2$. Then \cref{cor:Sidon} implies that $|A||B|=\frac{p-1}{d}-1,$ that is, 
\begin{align}\label{eq: AB}
d |A||B|=p-(d+1).    
\end{align}
On the other hand, \cref{thm:minlb} implies that we can find an absolute constant $C_d \in (0,1)$ such that
 $$
 C_d\sqrt{p}<\min\{|A|, |B|\}<\sqrt{p}.
 $$
It follows that $p-(d+1)$ has a divisor in the interval $(C_d\sqrt{p}, \sqrt{p})$. To summarize, if $p \in \mathcal{P}_d$, then we have 
$\tau (p-(d+1);C_d\sqrt{p},\sqrt{p}) \geq 1,$
where $\tau (n;y,z)$ denotes the number of divisors of $n$ in the interval $(y,z]$. Now, we use results by Ford \cite[Theorem 6]{F08} on the distribution of shift primes with a divisor in a given interval. Denote 
 \begin{align}
 H(x,y,z)& =\#\{1 \leq n \leq x: \tau (n;y,z) \geq 1\};\\
 P_d(x,y,z)&=\#\{p \leq x: \tau (p-(d+1);y,z) \geq 1\}.     
 \end{align}
 Setting $y=C_d\sqrt{x}/2$ and $z=\sqrt{x}$, \cite[Theorem 6]{F08} and \cite[Theorem 1, third case of (v)]{F08} imply that
 $$
 P_d(x,y,z) \ll \frac{H(x,y,z)}{\log x} \ll \frac{x}{\log x} u^\delta \bigg(\log \frac{2}{u}\bigg)^{-3/2}
 $$
 where $\delta=1-\frac{1+\log \log 2}{\log 2}$ and $u=\log (C_d/2)/\log y$. 
 It follows that as $x \to \infty$, we have $P_d(x,y,z)=o(x/\log x)$. Therefore, we have
 $$
 \#\{p \in \mathcal{P}_d: x/2\leq p\leq x\}\leq \#\{x/2\leq p\leq x: \tau (p-(d+1); C_d\sqrt{x}/2,\sqrt{x}) \geq 1\}=o(x/\log x).
 $$ 
 We conclude that as $x \to \infty$,
 \begin{align*}
 |\mathcal{P}_d \cap [0,x]|
 &=O(\sqrt{x})+\#\{p \in \mathcal{P}_d: \sqrt{x}\leq p\leq x\}\\
 &=O(\sqrt{x})+\sum_{0\leq j \leq (\log_2 x)/2} o\bigg(\frac{x/2^j}{\log (x/2^j)}\bigg)\\
 &=O(\sqrt{x})+\bigg(\sum_{0\leq j \leq (\log_2 x)/2} \frac{1}{2^j}\bigg)o\bigg(\frac{x}{\log x}\bigg)
 =o\bigg(\frac{x}{\log x}\bigg).     
 \end{align*}
\end{proof}

Using a similar argument, we can prove \cref{mainthm5}:
\begin{proof}[Proof of \textup{\cref{mainthm5}}]
    Consider the family of primes $p$ such that $p \equiv 1 \pmod d$ and $n$ is a $d$-th power modulo $p$. By a standard application of the Chebotarev density theorem, the density of such primes is given by $\frac{1}{[\mathbb{Q}(e^{2\pi i/d}, n^{1/d}):\mathbb{Q}]}$. Among the family of such primes $p$, we can repeat the same argument as in the proof of \cref{thm:almostall} to show that if  $(S_d(\F_p)-n) \setminus \{0\}$ admits a non-trivial multiplicative decomposition, then $p-(d+1)$ necessarily has a divisor which is ``close to'' $\sqrt{p}$. We remark that it is important to assume that $n$ is a $d$-th power modulo $p$, so that we can take advantage of \cref{cor:Sidon}. Similar to the proof of \cref{thm:almostall}, we can show that among the family of primes $p \equiv 1 \pmod d$, the property that $p-(d+1)$ has a divisor with the desired magnitude fails to hold for almost all $p$. This finishes the proof of the theorem.
\end{proof}

\begin{remark}
When $n$ is a fixed negative integer, one can obtain a similar result to \cref{mainthm5} following the idea of the above proof.
\end{remark}

\begin{remark}
\cref{mainthm5} essentially states if $d$ is fixed, $p \equiv 1 \pmod d$ is a prime, and  $\lambda \in S_d(\F_p)$, then it is very unlikely that we can decompose $(S_d(\F_p)-\lambda)\setminus \{0\}$ as the product of two subsets of $\F_p^*$ non-trivially. On the other hand, when $\lambda \notin S_d(\F_p)$, the above technique does not apply. Nevertheless, when $\lambda \notin S_d$ and we have two sets $A,B \subset \F_p^*$ such that $AB=(S_d-\lambda) \setminus \{0\}=S_d-\lambda$, \cref{thm: stepanovea} implies that
$$
|S_d|\leq |A||B|\leq |S_d|+\min \{|A|,|B|\}-1.
$$
In particular, we get the following non-trivial fact: if $|A|$ is fixed, then $|B|$ is also uniquely fixed. 
\end{remark}

\subsection{Applications to special multiplicative decompositions}\label{sec:specialMD}

In this subsection, we verify the ternary version of \cref{conjecture:MD} in a strong sense, which generalizes \cite[Theorem 2]{S14}.

Shkredov \cite[Theorem 3]{S20} showed if $G$ is a multiplicative subgroup of $\F_p$ with $1 \ll_{\epsilon} |G| \leq p^{6/7-\epsilon}$, then there is no $A \subset \F_p$ and $\xi \in \F_p^*$ such that $A/A=\xi G+1$. In fact, due to the analytic nature of the proof, he pointed out that his proof can be slightly modified to show something stronger, namely $A/A \neq (\xi G+1) \cup C$, as long as $C$ is small
(see also \cite[Remark 15]{S20}). The following corollary of \cref{thm: stepanovea} is of a similar flavor. 

\begin{corollary}
Let $p$ be a prime. If $G$ is a proper multiplicative subgroup of $\F_p$ with $|G|\geq 8$, and $\lambda, \xi \in \F_p^*$, then there is no $A \subset \F_p^*$ such that $AA=(\xi G-\lambda) \setminus \{0\}$.
\end{corollary}
\begin{proof}
We assume, otherwise, that $AA=(\xi G-\lambda) \setminus \{0\}$ for some $A\subset \mathbb{F}_p^{*}$. Then we observe that $aa'=a'a$ for each $a,a' \in A$, it follows that
$$
 |G|-1 \leq |AA|\leq \frac{|A|^2+|A|}{2}.    
$$
 Since $|G| \geq 8$, it follows that $|A| \geq 4$. 
Let $B=A/\xi$, and $\lambda'=\lambda/\xi$. Then we have $AB=(G-\lambda')\setminus \{0\}$ and thus \cref{thm: stepanovea} implies that
$|A|^2=|A||B|\leq |G|+|A|-1$. 
Comparing the above two inequalities, we obtain that
$$
|A|^2-|A| \leq |G|-1 \leq \frac{|A|^2+|A|}{2},
$$
which implies that $|A|\leq 3$, contradicting the assumption that $|A| \geq 4$.
\end{proof}

\begin{lemma}\label{lem:addcomb}
Let $A, B, C$ be nonempty subsets of $\F_q$ and let $\lambda \in \F_q^*$. Then $|ABC+\lambda|^2 \leq |AB+\lambda||BC+\lambda||CA+\lambda|$.
\end{lemma} 
\begin{proof}
It suffices to show $|ABC|^2 \leq |AB||BC||CA|$, which is a special case of \cite[Theorem 5.1]{R07} due to Ruzsa.   
\end{proof}

The following two theorems generalize S\'{a}rk\"{o}zy \cite[Theorem 2]{S14} and confirm the ternary version of \cref{conjecture:MD} in a strong form.

\begin{theorem}\label{thm:ternary}
There exists an absolute constant $M>0$, such that whenever $p$ is a prime, $G$ is a proper multiplicative subgroup of $\F_p$ with $|G|>M$, and $\lambda \in \F_p^*$,  there is no ternary multiplicative decomposition $ABC=(G-\lambda)\setminus \{0\}$ with $A,B,C \subset \F_p^*$ and $|A|,|B|,|C| \geq 2$.
\end{theorem}

\begin{proof}
Assume that there are sets $A,B,C \subset \F_p^*$ with $|A|,|B|,|C| \geq 2$, such that $ABC=(G-\lambda)\setminus \{0\}$ for some proper multiplicative subgroup $G$ of $\F_p$ and some $\lambda \in \F_p^*$.

Then we can write $(G-\lambda)\setminus \{0\}$ in three different ways: $A(BC), B(CA), C(AB)$, so that we can apply the results in previous sections to each of them. Note that \cref{1/2+o(1)} implies that $$|A|, |B|, |C| \geq |G|^{1/2+o(1)}.$$
On the other hand, \cref{thm: stepanovea} implies that
$$
|AB||C|,|BC||A|,|CA||B|\ll |G|.
$$
Therefore, from \cref{lem:addcomb} and the fact $|ABC| \in \{|G|, |G|-1\}$, we have
$$
|G|^2 |A||B||C|\ll |ABC|^2 |A||B||C| \ll (|AB||C|)(|BC||A|)(|CA||B|) \ll |G|^3.
$$
It follows that 
$$
|G|^{3/2+o(1)} \ll |A||B||C| \ll |G|,
$$
that is, $|G|\ll 1$, where the implicit constant is absolute. This completes the proof of the theorem.
\end{proof}

\begin{theorem}\label{thm: ternary multiplicative decomposition2}
Let $\epsilon>0$. There is a constant $Q=Q(\epsilon)$, such that for each prime power $q>Q$ and a divisor $d$ of $q-1$ with $2 \leq d \leq q^{1/10-\epsilon}$, there is no ternary multiplicative decomposition $ABC=(S_d(\F_q)-\lambda)\setminus \{0\}$ with $A,B,C \subset \F_q^*$, $|A|,|B|,|C| \geq 2$, and $\lambda \in \F_q^*$.     
\end{theorem}
\begin{proof}
The proof is similar to the proof of \cref{thm:ternary}. While \cref{1/2+o(1)} does not hold in the new setting (see \cref{remark:1/2+o(1)}), we can instead use \cref{thm:minlb}. If $ABC=(S_d(\F_q)-\lambda)\setminus \{0\}$, then \cref{thm:minlb} implies that 
$$
|A|, |B|, |C| \gg \frac{\sqrt{q}}{d}, \quad  |A||BC|,|BC||A|, |CA||B| \ll q.
$$
A similar computation leads to $d \gg q^{1/10}$, which implies that $q \ll_{\epsilon} 1$ since we assume that $d \leq q^{1/10-\epsilon}$.
\end{proof}

\section*{Acknowledgments}
The authors thank Andrej Dujella, Greg Martin, and J\'ozsef Solymosi for helpful discussions. The research of the second author was supported in part by an NSERC fellowship. The third author was supported by the KIAS Individual Grant (CG082701) at the Korea Institute for Advanced Study and the Institute for Basic Science (IBS-R029-C1). 

\bibliographystyle{abbrv}
\bibliography{references}

\appendix
\section{Algorithm and Computations}
\label{appendix}

We continue our discussion from the introduction on the following constant
$$\gamma_k=\limsup_{n \to \infty} \frac{M_k(n)}{\log n}.$$
It is implicit in \cite{DKM22} that $\gamma_k\leq 3 \phi(k)$. We also write $\nu_{k} = \frac{2k}{k-2} \eta_k \phi(k).$

Our main result, \cref{mainthm1}, implies that $\gamma_{k} \le \nu_{k}$.
In particular, in view of \cref{remark:6}, it follows that $\gamma_k \leq 6$ for all $k \geq 2$ \textcolor{black}{and $\gamma_k\leq 2+o(1)$ when $k \to \infty$}. In \cref{fig:graph1}, we pictorially compare our new bound $\nu_{k}$ with the bound $3 \phi(k)$ when $2 \le k \le 1000$.
\begin{figure}[H]
    \centering
    \includegraphics[scale=0.6]{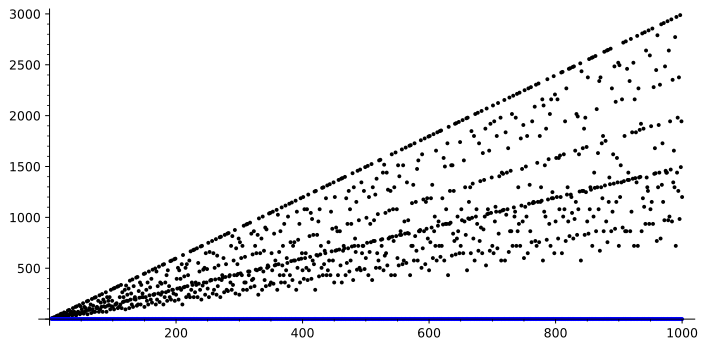}
    \caption{Comparison between the new bound $\nu_k$ and the bound $3\phi(k)$ in \cite{DKM22} when $2 \le k \le 1000$. The black dots denote $3\phi(k)$, and the blue dots denote $\nu_k$.}
    \label{fig:graph1}
\end{figure}
Recall that for $k \ge 2$, we defined the constant $\eta_k=\underset{\mathcal{I}}{\min} |\mathcal{I}|/T_\mathcal{I}^2,$
where the minimum is taken over all nonempty subsets $\mathcal{I}$ of 
$\{1 \leq i \leq k: \gcd(i,k)=1, \gcd(i-1,k)>1\},$ and $T_{\mathcal{I}}=\sum_{i \in \mathcal{I}} \sqrt{\gcd(i-1,k)}.$ 

To compute $\eta_k$, we use the following simple greedy algorithm with running time $O(k \log k)$. The observation is as follows. If $|\mathcal{I}|$ is fixed, our goal is to minimize $|\mathcal{I}|/T_{\mathcal{I}}^2$. Thus, we should choose those residue classes $i \pmod k$ with $\gcd(i-1,k)$ as large as possible to maximize $T_{\mathcal{I}}$. Then, we can sort these gcds in decreasing order, and when $|\mathcal{I}|$ is fixed, we pick those residue classes corresponding to the largest $|\mathcal{I}|$ gcds. The following is a precise description of the algorithm:
\begin{algorithm}
Let $k \ge 2$. We follow the notations defined in \cref{subsec: size of Mk}.
\begin{enumerate}
\item[\textit{Step 1.}] Let $\mathcal{A}=\{1 \leq i \leq k: \gcd(i,k)=1, \gcd(i-1,k)>1\}$. We list the elements of $\mathcal{A}$ by $\{a_{1},a_{2},\ldots\}$ such that $\gcd(a_j -1,k)$ is decreasing by using a sorting algorithm. 
\item[\textit{Step 2.}] Set $I_{r}=\{a_{1},\ldots,a_{r}\}$, $T_{I_{r}}=\sum_{i \in I_{r}} \sqrt{\gcd(i-1,k)}$, and $\xi_{I_r}=|I_r|/T_{I_r}^2$.
\item[\textit{Step 3.}] Return $\eta_{k}=\min_{r}\xi_{I_r}$ and terminate the algorithm.
\end{enumerate}
\end{algorithm}
Note that the running time of the above algorithm is $O(k \log k)$: sorting takes $O(k \log k)$ time, while other steps take linear time. 

Next, we also consider the minimum value $m_{k}$ of the upper bounds $\{ \nu_{i} \colon 2 \le i \le k\}$ for each $k \ge 2$.
\cref{table: mini} shows the values of $m_{k}$ for $2 \le k \leq 1{,}200{,}000$ when they are changed. 

\begin{table}[H]
\begin{tabular}{c||c}
$k$ & $m_{k}$ \\ \hline
 2 &  2.00000 \\ \hline 
  4 &  1.37258 \\ \hline
  6 &  0.80385 \\ \hline
  8 &  0.72776 \\ \hline
  12 &  0.44134 \\ \hline
  24 &  0.31910 \\ \hline
  36 &  0.29027 \\ \hline
  48 &  0.25836 \\ \hline
  60 &  0.21636 \\ \hline
  120 &  0.16570 \\ 
\end{tabular}
\quad 
\begin{tabular}{c||c}
$k$ & $m_{k}$ \\ \hline
  180 &  0.15191 \\ \hline
  240 &  0.13876 \\ \hline
  360 &  0.11708 \\ \hline
  720 &  0.09693 \\ \hline
  840 &  0.09266 \\ \hline
  1260 &  0.08465 \\ \hline
  1440 &  0.08445 \\ \hline
  1680 &  0.07624 \\ \hline
  2520 &  0.06465 \\ \hline
  5040 &  0.05317 \\ 
\end{tabular}
\quad 
\begin{tabular}{c||c}
$k$ & $m_{k}$ \\ \hline  
  7560 &  0.05171 \\ \hline
  10080 &  0.04592 \\ \hline
  15120 &  0.04252 \\ \hline
  20160 &  0.04111 \\ \hline
  25200 &  0.03887 \\ \hline
  27720 &  0.03665 \\ \hline
  30240 &  0.03647 \\ \hline
  50400 &  0.03343 \\ \hline
  55440 &  0.02997 \\ \hline
  83160 &   0.02877 \\  
\end{tabular}
\quad 
\begin{tabular}{c||c}
$k$ & $m_{k}$ \\ \hline
  110880 &   0.02574 \\  \hline
  166320 &  0.02343 \\ \hline
  221760 &  0.02280 \\ \hline
  277200 &  0.02138 \\ \hline
   332640 &  0.02008 \\ \hline
   498960 &  0.01985 \\ \hline
   554400 &  0.01827 \\ \hline
    665280 &  0.01774 \\ \hline
     720720 &  0.01654 \\ \hline 
      1081080 & 0.01587 \\ 
\end{tabular}

\bigskip
\caption{The minimum $m_{k}$ of the upper bounds $\{\nu_{i} \colon 1 \le i \le k\}$ for $2 \le k \leq 1{,}200{,}000$.}\label{table: mini}
\end{table}

We also report our computations on $\nu_k$ for $2 \leq k \leq 201$ in the following table.

\begin{table}[]
\begin{tabular}{c||c}
$k$ & $\nu_k$ \\ \hline
2 &  2.0000 \\ \hline
3 & 4.0000 \\ \hline
 4 & 1.3726 \\ \hline
 5 & 2.6667 \\ \hline
 6 & 0.8038 \\ \hline
 7 & 2.4000 \\ \hline
 8 & 0.7278 \\ \hline
 9 & 1.1077 \\ \hline
 10 & 0.7295 \\ \hline
 11 & 2.2222 \\ \hline
 12 & 0.4413 \\ \hline
 13 & 2.1818 \\ \hline
 14 & 0.7185 \\ \hline
 15 & 0.7222 \\ \hline
 16 & 0.5383 \\ \hline
 17 & 2.1333 \\ \hline
 18 & 0.4522 \\ \hline
 19 & 2.1176 \\ \hline
 20 & 0.4450 \\ \hline
 21 & 0.7355 \\ \hline
 22 & 0.7251 \\ \hline
 23 & 2.0952 \\ \hline
 24 & 0.3191 \\ \hline
 25 & 1.1180 \\ \hline
 26 & 0.7313 \\ \hline
 27 & 0.7508 \\ \hline
 28 & 0.4552 \\ \hline
 29 & 2.0741 \\ \hline
 30 & 0.3351 \\ \hline
 31 & 2.0690 \\ \hline
 32 & 0.4555 \\ \hline
 33 & 0.7709 \\ \hline
 34 & 0.7438 \\ \hline
 35 & 0.7311 \\ \hline
 36 & 0.2903 \\ \hline
 37 & 2.0571 \\ \hline
 38 & 0.7497 \\ \hline
 39 & 0.7873 \\ \hline
 40 & 0.3353 \\ \hline
 41 & 2.0513 \\ 
 \end{tabular}
\quad
\begin{tabular}{c||c}
$k$ & $\nu_k$ \\ \hline
42 & 0.3465 \\ \hline
 43 & 2.0488 \\ \hline
 44 & 0.4739 \\ \hline
 45 & 0.4581 \\ \hline
 46 & 0.7604 \\ \hline
 47 & 2.0444 \\ \hline
 48 & 0.2584 \\ \hline
 49 & 1.1716 \\ \hline
 50 & 0.4974 \\ \hline
 51 & 0.8163 \\ \hline
 52 & 0.4818 \\ \hline
 53 & 2.0392 \\ \hline
 54 & 0.3596 \\ \hline
 55 & 0.7678 \\ \hline
 56 & 0.3486 \\ \hline
 57 & 0.8290 \\ \hline
 58 & 0.7742 \\ \hline
 59 & 2.0351 \\ \hline
 60 & 0.2164 \\ \hline
 61 & 2.0339 \\ \hline
 62 & 0.7783 \\ \hline
 63 & 0.4746 \\ \hline
 64 & 0.4124 \\ \hline
 65 & 0.7860 \\ \hline
 66 & 0.3643 \\ \hline
 67 & 2.0308 \\ \hline
 68 & 0.4950 \\ \hline
 69 & 0.8515 \\ \hline
 70 & 0.3548 \\ \hline
 71 & 2.0290 \\ \hline
 72 & 0.2171 \\ \hline
 73 & 2.0282 \\ \hline
 74 & 0.7892 \\ \hline
 75 & 0.5005 \\ \hline
 76 & 0.5006 \\ \hline
 77 & 0.7644 \\ \hline
 78 & 0.3713 \\ \hline
 79 & 2.0260 \\ \hline
 80 & 0.2730 \\ \hline
 81 & 0.6359 \\ 
 \end{tabular}
\quad
\begin{tabular}{c||c}
$k$ & $\nu_k$ \\ \hline
82 & 0.7956 \\ \hline
 83 & 2.0247 \\ \hline
 84 & 0.2263 \\ \hline
 85 & 0.8195 \\ \hline
 86 & 0.7985 \\ \hline
 87 & 0.8796 \\ \hline
 88 & 0.3682 \\ \hline
 89 & 2.0230 \\ \hline
 90 & 0.2239 \\ \hline
 91 & 0.7794 \\ \hline
 92 & 0.5103 \\ \hline
 93 & 0.8877 \\ \hline
 94 & 0.8040 \\ \hline
 95 & 0.8346 \\ \hline
 96 & 0.2261 \\ \hline
 97 & 2.0211 \\ \hline
 98 & 0.5376 \\ \hline
 99 & 0.5038 \\ \hline
 100 & 0.3344 \\ \hline
 101 & 2.0202 \\ \hline
 102 & 0.3827 \\ \hline
 103 & 2.0198 \\ \hline
 104 & 0.3757 \\ \hline
 105 & 0.3626 \\ \hline
 106 & 0.8114 \\ \hline
 107 & 2.0190 \\ \hline
 108 & 0.2355 \\ \hline
 109 & 2.0187 \\ \hline
 110 & 0.3768 \\ \hline
 111 & 0.9094 \\ \hline
 112 & 0.2864 \\ \hline
 113 & 2.0180 \\ \hline
 114 & 0.3874 \\ \hline
 115 & 0.8621 \\ \hline
 116 & 0.5220 \\ \hline
 117 & 0.5160 \\ \hline
 118 & 0.8179 \\ \hline
 119 & 0.8087 \\ \hline
 120 & 0.1657 \\ \hline
 121 & 1.2615 \\ 
 \end{tabular}
\quad
\begin{tabular}{c||c}
$k$ & $\nu_k$ \\ \hline
122 & 0.8200 \\ \hline
 123 & 0.9220 \\ \hline
 124 & 0.5254 \\ \hline
 125 & 0.8820 \\ \hline
 126 & 0.2335 \\ \hline
 127 & 2.0160 \\ \hline
 128 & 0.3877 \\ \hline
 129 & 0.9278 \\ \hline
 130 & 0.3869 \\ \hline
 131 & 2.0155 \\ \hline
 132 & 0.2413 \\ \hline
 133 & 0.8226 \\ \hline
 134 & 0.8256 \\ \hline
 135 & 0.3709 \\ \hline
 136 & 0.3878 \\ \hline
 137 & 2.0148 \\ \hline
 138 & 0.3955 \\ \hline
 139 & 2.0146 \\ \hline
 140 & 0.2400 \\ \hline
 141 & 0.9387 \\ \hline
 142 & 0.8291 \\ \hline
 143 & 0.7812 \\ \hline
 144 & 0.1809 \\ \hline
 145 & 0.8976 \\ \hline
 146 & 0.8307 \\ \hline
 147 & 0.5506 \\ \hline
 148 & 0.5342 \\ \hline
 149 & 2.0136 \\ \hline
 150 & 0.2468 \\ \hline
 151 & 2.0134 \\ \hline
 152 & 0.3928 \\ \hline
 153 & 0.5366 \\ \hline
 154 & 0.3772 \\ \hline
 155 & 0.9081 \\ \hline
 156 & 0.2471 \\ \hline
 157 & 2.0129 \\ \hline
 158 & 0.8353 \\ \hline
 159 & 0.9532 \\ \hline
 160 & 0.2385 \\ \hline
 161 & 0.8485 \\ 
 \end{tabular}
\quad
\begin{tabular}{c||c}
$k$ & $\nu_k$ \\ \hline
162 & 0.3140 \\ \hline
 163 & 2.0124 \\ \hline
 164 & 0.5392 \\ \hline
 165 & 0.3857 \\ \hline
 166 & 0.8382 \\ \hline
 167 & 2.0121 \\ \hline
 168 & 0.1737 \\ \hline
 169 & 1.2965 \\ \hline
 170 & 0.4049 \\ \hline
 171 & 0.5453 \\ \hline
 172 & 0.5416 \\ \hline
 173 & 2.0117 \\ \hline
 174 & 0.4053 \\ \hline
 175 & 0.5104 \\ \hline
 176 & 0.3077 \\ \hline
 177 & 0.9660 \\ \hline
 178 & 0.8422 \\ \hline
 179 & 2.0113 \\ \hline
 180 & 0.1519 \\ \hline
 181 & 2.0112 \\ \hline
 182 & 0.3854 \\ \hline
 183 & 0.9700 \\ \hline
 184 & 0.4014 \\ \hline
 185 & 0.9365 \\ \hline
 186 & 0.4080 \\ \hline
 187 & 0.8001 \\ \hline
 188 & 0.5459 \\ \hline
 189 & 0.3843 \\ \hline
 190 & 0.4129 \\ \hline
 191 & 2.0106 \\ \hline
 192 & 0.2075 \\ \hline
 193 & 2.0105 \\ \hline
 194 & 0.8471 \\ \hline
 195 & 0.3948 \\ \hline
 196 & 0.3652 \\ \hline
 197 & 2.0103 \\ \hline
 198 & 0.2493 \\ \hline
 199 & 2.0102 \\ \hline
 200 & 0.2517 \\ \hline
 201 & 0.9811 \\ 
 \end{tabular}

\bigskip
\caption{The upper bound $\nu_k$ of $\gamma_k$ when $2 \le k \le 201$} 
\label{table: the upper bound1}
\end{table}

\end{document}